\documentclass[hidelinks,onefignum,onetabnum,nolinenumbers]{siamart}
\usepackage[T1]{fontenc}          
\usepackage{mathtools}             
\usepackage{bm}                    
\usepackage{mathrsfs}              
\usepackage{lipsum}
\usepackage{amsfonts}
\usepackage{graphicx,slashbox}
\usepackage{hyperref}
\usepackage{algorithm}
\usepackage{algpseudocode}
\newtheorem{example}{Example}

\newtheorem{lem}{Lemma}
\newtheorem{prop}{Proposition}
\newtheorem{thm}{Theorem}
\ifpdf
  \DeclareGraphicsExtensions{.eps,.pdf,.png,.jpg}
\else
  \DeclareGraphicsExtensions{.eps}
\fi

\newsiamremark{remark}{Remark}
\newsiamremark{hypothesis}{Hypothesis}
\crefname{hypothesis}{Hypothesis}{Hypotheses}
\newsiamthm{claim}{Claim}
\graphicspath{{./figs/}} 
\usepackage{algorithm}             
\usepackage{algpseudocode}         

\usepackage{graphicx}              
\usepackage{float}                 
\usepackage{subcaption}            
\usepackage{booktabs}              
\usepackage{multirow}              

\usepackage{enumerate}             
\usepackage{enumitem}              
\usepackage{setspace}              

\usepackage{array}                 
\newcommand{\e}{\mathrm{e}}        

\newcommand{\nrm}[1]{\left\| #1 \right\|}

\newcommand{\eipx}[2]{\left[ #1 , #2 \right]_{\rm x}}
\newcommand{\eipy}[2]{\left[ #1 , #2 \right]_{\rm y}}
\newcommand{\eipz}[2]{\left[ #1 , #2 \right]_{\rm z}}

 \newcommand{\hf}{\frac{1}{2}}
 
\title{An Efficient Unconditionally Energy-Stable Numerical Scheme for Bose--Einstein Condensate\thanks{
The second author was partially supported by National Science Foundation of the United States (Grant No. DMS-2309548).	The third author was partially supported by National Natural Science Foundation of China (Grant No. 12422116), Guangdong Basic and Applied Basic Research Foundation (Grant No. 2023A1515012199), Shenzhen Science and Technology Innovation Program (Grant No. JCYJ20220530143803007, RCYX20221008092843046), Guangdong Provincial Key Laboratory of Mathematical Foundations for Artificial Intelligence (2023B1212010001), and Hetao Shenzhen-Hong Kong Science and Technology Innovation Cooperation Zone Project (No. HZQSWS-KCCYB-2024016).  }}
\author{Jing Guo\thanks{School of Mathematics and Statistics, Guangdong University of Technology, Guangdong, Guangzhou 510006, China (\email{jingguo@gdut.edu.cn}).}\and
Cheng Wang\thanks{Corresponding author. Department of Mathematics; University of Massachusetts; North Dartmouth, MA 02747, USA (\email{cwang1@umassd.edu}).}\and
Dong Wang  	\thanks{ School of Science and Engineering, The Chinese University of Hong Kong, Shenzhen, Guangdong 518172, China; Shenzhen International Center for Industrial and Applied Mathematics, Shenzhen Research Institute of Big Data, Guangdong 518172, China (\email{wangdong@cuhk.edu.cn}).}
}
\begin{document}
\maketitle
\begin{abstract}
A numerical framework is proposed and analyzed for computing the ground state of Bose--Einstein condensates. A gradient flow approach is developed, incorporating both a Lagrange multiplier to enforce the $L^2$ conservation and a free energy dissipation. An explicit approximation is applied to the chemical potential, combined with an exponential time differencing (ETD) operator to the diffusion part, as well a stabilizing operator, to obtain an intermediate numerical profile. Afterward, an $L^2$ normalization is applied at the next numerical stage. A theoretical analysis reveals a free energy dissipation under a maximum norm bound assumption for the numerical solution, and such a maximum norm bound could be recovered by a careful convergence analysis and error estimate. In the authors' knowledge, the proposed method is the first numerical work that preserves the following combined theoretical properties: (1) an explicit computation at each time step, (2) unconditional free energy dissipation, (3) $L^2$ norm conservation at each time step, (4) a theoretical justification of convergence analysis and optimal rate error estimate. Comprehensive numerical experiments validate these theoretical results, demonstrating excellent agreement with established reference solutions.

\end{abstract}


\section{Introduction}
Bose--Einstein condensation (BEC) describes a macroscopic quantum phenomenon in which many bosons occupy the lowest quantum state, forming a coherent matter wave. It provides a versatile platform for studying quantum many-body physics \cite{Bloch2008, Leggett2001}, quantum simulation \cite{Georgescu2014}, and precision sensing \cite{Cronin2009}. A theoretical understanding of BEC relies on the Gross--Pitaevskii theory, where the complex valued, ground-state wave function $\phi$ is obtained by minimizing
\begin{equation} \label{eng_orig}
    E(\phi) = \int_{\Omega} \left( \frac{1}{2} |\nabla \phi|^2 + V({\bf x}) |\phi|^2 
    + \frac{\beta}{2} |\phi|^4  \right) \, d {\bf x},
\end{equation}
subject to the normalization constraint
\begin{equation}\label{cons}
\int_{\Omega} |\phi(\mathbf{x})|^2 d\mathbf{x} = 1,
\end{equation}
where $\Omega \subset \mathbb{R}^d$. The corresponding Euler--Lagrange equation, known as the Gross--Pitaevskii equation (GPE), is given by
\begin{equation}\label{eig_prob}
    \mu (\phi) = -\frac{1}{2}\Delta\phi + V\phi + \beta|\phi|^2\phi,
\end{equation}
governs the equilibrium configuration of the condensate.

The scientific computation of BEC has attracted considerable research attention due to its fundamental importance in quantum physics. A particularly prominent approach is the gradient flow method, which employs imaginary-time evolution to implement gradient descent for the energy functional \cite{BaoCai,BaoDu04,Bao2008,BaoWang07,CaiLiu, Faou2018, LiuCai, Zhuang2019} and  a projection step is adopted   to maintain the mass constraint \eqref{cons}. 

Optimization-based approaches offer an alternative computational paradigm by directly minimizing the energy functional on an appropriately defined constraint manifold, including the Riemannian gradient method \cite{YinHuaCai, YinHuaZha}, the Riemannian conjugate gradient method \cite{ShuTang}, and the Riemannian Newton method \cite{JiaAndWen, TianCai}, etc. 

Eigenvalue-based approaches reformulate the computation of ground states as the solution of a stationary nonlinear eigenvalue problem. Well-established computational techniques in this category include the J-method \cite{RobPatDan},  finite element discretization \cite{Chen2011a, ChenHeZh}, and mixed finite element formulation \cite{Gallistl2025}.

While the computation of ground states has been extensively investigated from an algorithmic perspective, the corresponding numerical analysis remains relatively limited and presents significant challenges. 
 For the normalized gradient flow methods, a local convergence analysis was established in \cite{Faou2018}, with a global convergence result provided  in \cite{Henning2020}.  Additionally, the authors in \cite{ChenHeZh, Henning2023, HenYad} derived error estimates for the Gross--Pitaevskii eigenvalue problems.


The present work makes two primary contributions.  First, we develop an explicit, norm-preserving numerical scheme that maintains computational efficiency. In the numerical design, an explicit approximation is applied to the nonlinear term and the external confinement term, while the diffusion part is computed using an exponential time differencing (ETD) formula, to obtain an intermediate numerical profile. In particular, an artificial stabilization term is introduced in this intermediate stage, for the sake of a theoretical justification of energy dissipation, as will be proved in the later part. Such an artificial stabilization term would not cause additional computational cost, since it could be combined with the temporal discretization part. Afterward, an $L^2$ normalization is performed at the next numerical stage. As a second contribution, rigorous proofs are provided to the proposed numerical scheme. A theoretical analysis reveals a free energy dissipation under a maximum norm bound assumption for the numerical solution, and such a maximum norm bound could be recovered by a careful convergence analysis and error estimate. In turn, the proposed method preserves the following theoretical properties: (1) an explicit computation at each time step, (2) unconditional free energy dissipation, (3) $L^2$ norm conservation at each time step, (4) a theoretical justification of convergence analysis and optimal rate error estimate. To the best of our knowledge, this is the first such numerical work that preserves the combined four theoretical properties in the computation of the BEC problem. 


The paper is structured as follows. The mathematical formulation is reviewed in Section~\ref{sec: numerical scheme}, and the numerical scheme is proposed. The energy stability analysis is provided in Section \ref{sec: energy stability}, while the convergence analysis and optimal rate error estimate is established in Section \ref{sec: convergence analysis}. A few numerical examples are presented in Section \ref{sec: numerical results}. Finally, some concluding remarks are made in Section \ref{sec: conclusion}.

\section{Mathematical formulation and the numerical scheme} \label{sec: numerical scheme} 

\subsection{Gradient flow with a Lagrange multiplier}
The normalized gradient flow with Lagrange multiplier  takes the form of 
\begin{equation}\label{grad_flow_cont}
\frac{\partial \phi}{\partial t} = \frac{1}{2}\Delta \phi - V(\mathbf{x})\phi - \beta|\phi|^2\phi + \lambda(t)\phi,
\end{equation}
where the time-dependent Lagrange multiplier \cite[(2.14)]{BaoDu04} 
\begin{equation}\label{lambda_cont}
\lambda(t) =\frac{ ( -\frac{1}{2}\Delta\phi + V(\mathbf{x})\phi + \beta|\phi|^2\phi, \phi )}{\left\Vert \phi\right\Vert^2} 
= \frac{ \frac12 \| \nabla \phi \|^2 + ( V ({\bf x}) \phi + \beta | \phi |^2 \phi , \phi)}{\| \phi \|^2} , 
\end{equation}
enforces the $L^2$ normalization constraint. This system preserves the $L^2$ norm, and an energy dissipation becomes available: 
\begin{equation}\label{energy_decay}
\frac{d}{dt}E(\phi) \leq 0.
\end{equation}

\subsection{A proposed numerical scheme}
The standard centered finite difference discretization is used to compute the ground state of \eqref{eng_orig} with $\beta>0$  on the computational domain $\Omega = (0, 1)^3$,  with the
periodic boundary condition taken into consideration.  An extension to other boundary conditions, such as the homogeneous Neumann one, would be straightforward. For simplicity of presentation, a uniform spatial mesh is taken, with $\Delta x = \Delta y = \Delta z = h = \frac{1}{N}$, with $N \in\mathbb{N}$. Moreover, $f_{i,j,k}$ represents the numerical value of $f$ at the regular numerical mesh points $( i h, j h, k h )$, and a discrete space ${\mathcal C}_{\rm per}$ is introduced as
	$$
{\mathcal C}_{\rm per} := \left\{ f = (f_{i,j,k} ) \,|\,  f_{i,j,k} = f_{i+\alpha N,j+\beta N, k+\gamma N}, \ \forall \, i,j,k,\alpha,\beta,\gamma\in \mathbb{Z} \right\}.
	$$
In turn, the discrete difference operators are evaluated at $( (i + \frac12) h , j h , k h)$, $( i h, ( j+\frac12) h, k h)$ and $( i h , j h , (k+\frac12)h )$, respectively:
	\begin{align*}
& 
D_x f_{i+\hf,j,k} := \frac{1}{h} (f_{i+1,j,k} - f_{i,j,k}  ), \quad
D_y f_{i,j+\hf,k} := \frac{1}{h} (f_{i,j+1,k} - f_{i,j,k}  ) ,
	\\
& 
D_z f_{i,j,k+\hf} := \frac{1}{h} (f_{i,j,k+1} - f_{i,j,k}  ) .
	\end{align*}
For a vector function $\vec{f} = ( f^x , f^y , f^z)^T$ with $f^x$, $f^y$, $f^z$ evaluated at $( (i+\frac12)h, j h , k h)$, $( i h, (j+\frac12) h, k h)$, $( i h , j h , ( k+\frac12)h )$, respectively, the corresponding average and difference operators are given by 
	\begin{align*}
& a_x f^x_{i, j, k} := \frac{1}{2} \big(f^x_{i+\hf, j, k} + f^x_{i-\hf, j, k} \big),	 \quad
 d_x f^x_{i, j, k} := \frac{1}{h}\big(f^x_{i+\hf, j, k} - f^x_{i-\hf, j, k} \big),
          \\
 & a_y f^y_{i,j, k} := \frac{1}{2} \big(f^y_{i,j+\hf, k} + f^y_{i,j-\hf, k} \big),	 \quad
d_y f^y_{i,j, k} := \frac{1}{h} \big(f^y_{i,j+\hf, k} - f^y_{i,j-\hf, k} \big),
	\\
& a_z f^z_{i,j,k} := \frac{1}{2} \big(f^z_{i, j,k+\hf} + f^z_{i, j, k-\hf} \big),   \quad
 d_z f^z_{i,j, k} := \frac{1}{h} \big(f^z_{i, j,k+\hf} - f^z_{i, j,k-\hf} \big) .
	\end{align*}
The average operators $A_x$, $A_y$ and $A_z$, evaluated at the staggered mesh points $(i+\hf, j,k)$, $(i,j+\hf,k)$ and $(i,j,k+\hf)$, respectively, could be similarly defined. Subsequently,  the discrete divergence turns out to be
\begin{equation*}
\nabla_h\cdot \big( \vec{f} \big)_{i,j,k} = ( d_x f^x )_{i,j,k}  + ( d_y f^y )_{i,j,k} + ( d_z f^z )_{i,j,k} .
\end{equation*}
In particular, if $\vec{f} = \nabla_h \phi = ( D_x \phi , D_y \phi, D_z \phi)^T$ for certain scalar grid function $\phi$, the corresponding divergence becomes the standard Laplacian operator $\Delta_h f = \nabla_h \cdot \nabla_h = D_x^2 + D_y^2 + D_z^2$. 

For two cell-centered, complex valued grid functions $f$ and $g$, the discrete $L^2$ inner product and the associated $\ell^2$ norm are defined as
\begin{equation*}
 \langle f , g \rangle := h^3 \sum_{i,j,k=1}^N \, ( f_{i,j,k} )^c g_{i,j,k} , \quad
    \| f \|_2 := (\left\langle f , f\right\rangle )^\frac12 .
\end{equation*}
For simplicity, we only focus on the real part of a discrete inner product throughout this article. In other words, two inner products have equivalent values if their real parts are equal. Meanwhile, the mean zero space is introduced as $\mathring{\mathcal C}_{\rm per}:=\big\{ f \in {\mathcal C}_{\rm per} \, \big| \, \overline{f} :=  \frac{1}{| \Omega|} \langle f , 1 \rangle = 0 \big\}$. Similarly, for two vector grid functions $\vec{f} = ( f^x , f^y , f^z )^T$ and $\vec{g} = ( g^x , g^y , g^z )^T$ with $f^x$ ($g^x$), $f^y$ ($g^y$), $f^z$ ($g^z$) evaluated at $( (i+\frac12)h, j h , (k+\frac12) h)$, $( i  h, (j+\frac12) h, k h)$, $( i h , j h , ( k+\frac12)h )$, respectively, the corresponding discrete inner product becomes
	\begin{align*}
	  &
   \langle \vec{f} , \vec{g} \rangle : = \eipx{f^x}{g^x}	+ \eipy{f^y}{g^y} + \eipz{f^z}{g^z}, 	
\\
  & \eipx{f^x}{g^x} := \langle a_x (f^x g^x) , 1 \rangle , \, \, \,
   \eipy{f^y}{g^y} := \langle a_y (f^y g^y) , 1 \rangle , \, \, \,
   \eipz{f^z}{g^z} := \langle a_z (f^z g^z) , 1 \rangle.
	\end{align*}
In addition to the $\ell^2$ norm, the discrete maximum norm and the $\ell^p$ ($1 \le p < + \infty$) norm are introduced as  $\nrm{f}_\infty := \max_{1\le i,j,k\le N}\left| f_{i,j,k}\right|$, $\| f \|_p = ( \langle | f |^p , 1 \rangle )^\frac{1}{p}$. Moreover, the discrete $H_h^1$ and $H_h^2$ norms are defined as 
\begin{equation*} 
\begin{aligned} 
  & 
\nrm{ \nabla_h f}_2^2 : = \langle \nabla_h f , \nabla_h f \rangle , \quad 
\nrm{f}_{H_h^1}^2 : =  \nrm{f}_2^2+ \nrm{ \nabla_h f}_2^2 , 
\\
  & 
\| f \|_{H_h^2}^2 := \| f \|_{H_h^1}^2 + \| D_x^2 f \|_2^2 + \| D_y^2 f \|_2^2 
  + \| D_z^2 f \|_2^2 + \| \Delta_h f \|_2^2 . 
\end{aligned} 
\end{equation*} 
For any $\psi, \phi \in {\mathcal C}_{\rm per}$ and any $\vec{f}$, the following summation-by-parts formulas are valid; see the related derivations in \cite{guan14a, guo16, wang11a, wise09}, etc: 
	\begin{equation}
\langle \psi , \nabla_h\cdot\vec{f} \rangle = - \langle \nabla_h \psi ,  \vec{f} \rangle, \, \, 
 \langle \psi, \Delta_h \phi \rangle = - \langle \nabla_h \psi ,  \nabla_h\phi \rangle , \, \, 
 \langle \psi, \Delta_h^2 \phi \rangle = \langle \Delta_h \psi ,  \Delta_h \phi \rangle . 	
  \label{summation by parts-1} 
 \end{equation}

To develop an efficient numerical solver for the normalized gradient flow system \eqref{grad_flow_cont}, we introduce an ETD operator, $\mathcal{G}_h = (- \frac12 \tau \Delta_h)^{-1}(I - {\rm e}^{\frac12 \tau\Delta_h})$. In more details,  for any $f \in {\mathcal C}_{\rm per}$ with the following discrete Fourier expansion:
\begin{equation}
  f_{i,j,k} = \sum_{\ell,m,n=-K}^K \hat{f}_{\ell, m,n} {\rm e}^{2 \pi \mathrm{i} ( \ell x_i + m y_j + n z_k)} ,  
  \label{Fourier-1}
\end{equation}
the associated ETD operator can be represented as
\begin{equation} 
  ( {\cal G}_h f )_{i,j,k} = \sum_{\ell, m,n=-K}^K  \frac{1
  - {\rm e}^{- \frac12 \tau \lambda_{\ell,m.n}} }{\frac12 \tau \lambda_{\ell,m,n}}
  \hat{f}_{\ell,m,n} {\rm e}^{2 \pi \mathrm{i} ( \ell x_i + m y_j + n z_k)} ,  
  \label{G_operator} 
\end{equation}
with $\lambda_{\ell,m,n} = \frac{4}{h^2} \big( \sin^2 (\ell \pi h) + \sin^2 (m \pi h)  + \sin^2 (n \pi h) \big)$. This operator serves as an $O (\tau + h^2)$ numerical approximation to the identity operator while maintaining crucial numerical stability properties. Moreover, since all the eigenvalues of ${\cal G}_h$ are non-negative, a natural definition of ${\cal G}^\frac12_h$ is given by
\begin{equation}
  ( {\cal G}_h^\frac12 f )_{i,j,k} = \sum_{\ell, m,n=-K}^K  \Big( \frac{1 - {\rm e}^{- \frac12 \tau \lambda_{\ell,m,n}} }{\frac12 \tau \lambda_{\ell,m,n}} \Big)^{\frac12}
  \hat{f}_{\ell,m,n} {\rm e}^{2 \pi \mathrm{i} ( \ell x_i + m y_j + n z_k)}  .
    \label{Fourier-3}
\end{equation}
It is clear that the operator ${\cal G}_h$ is commutative with any discrete differential operator, and the following summation by parts formula is valid:
\begin{equation}
   \langle f , {\cal G}_h g \rangle = \langle {\cal G}_h f ,  g \rangle 
   = \langle {\cal G}_h^\frac12  f , {\cal G}_h^\frac12 g \rangle , \quad \forall f , g \in {\cal C}_{\rm per} . 
     \label{Fourier-4}
\end{equation}

With the help of this ETD operator, we are able to construct the following explicit numerical scheme:
\begin{equation}\label{num_scheme}
\begin{aligned}
\frac{\tilde{\phi}^{n+1} - \phi^n}{\tau} &= \frac{1}{2}\mathcal{G}_h \Delta_h \phi^n - V(\mathbf{x})\phi^n - \beta|\phi^n|^2\phi^n + \lambda^n \phi^n - A (\tilde{\phi}^{n+1} - \phi^n ) , \\
\phi^{n+1} &= \frac{\tilde{\phi}^{n+1}}{\|\tilde{\phi}^{n+1}\|_2},
\end{aligned} 
\end{equation}
where the first equation in \eqref{num_scheme} represents the gradient flow evolution and the second equation corresponds to a standard renormalization. The Lagrange multiplier $\lambda^n$ is computed by
\begin{equation}\label{lambda_disc}
\lambda^n = \frac12 \| {\cal G}_h^\frac12 \nabla_h \phi^n \|_2^2 + \langle V(\mathbf{x})\phi^n + \beta|\phi^n|^2\phi^n, \phi^n \rangle.
\end{equation}

\section{Energy stability analysis} \label{sec: energy stability} 

The following equalities and inequalities will be useful in the energy stability analysis. 

\begin{lem}  \label{lem 1: preliminary}
The following estimates are valid: 
\begin{align} 
  & 
  \frac{1}{\tau} \| \tilde{\phi}^{n+1} - \phi^n \|_2^2 
  - \frac12 \langle \mathcal{G}_h \Delta_h \phi^n , \tilde{\phi}^{n+1} - \phi^n \rangle  
  \ge \frac14 ( \| {\cal G}_h^\frac12 \nabla_h \tilde{\phi}^{n+1} \|_2^2 
  - \| {\cal G}_h^\frac12 \nabla_h \phi^n \|_2^2 ) , 
  \label{lem 1-1-1} 
\\
  & 
  \langle \tilde{\phi}^{n+1} - \phi^n , \phi^n \rangle = 0 ,  \label{lem 1-1-2} 
\\
  & 
  \| \tilde{\phi}^{n+1} \|_2 \ge \| \phi^n \|_2 = 1 .  \label{lem 1-1-3} 
\end{align} 
\end{lem} 

\begin{proof} 
By the summation by parts formulas~\eqref{summation by parts-1} and \eqref{Fourier-4}, we see that 
\begin{equation} 
\begin{aligned} 
  & 
  - \frac12 \langle \mathcal{G}_h \Delta_h \phi^n , \tilde{\phi}^{n+1} - \phi^n \rangle 
  = \frac12 \langle {\cal G}_h^\frac12 \nabla_h \phi^n , 
  {\cal G}_h^\frac12 \nabla_h ( \tilde{\phi}^{n+1} - \phi^n ) \rangle   
\\
  =& 
   \frac14 (  \|  {\cal G}_h^\frac12 \nabla_h \tilde{\phi}^{n+1} \|_2^2 
   - \|  {\cal G}_h^\frac12 \nabla_h \phi^n \|_2^2  
  - \| {\cal G}_h^\frac12 \nabla_h ( \tilde{\phi}^{n+1} - \phi^n ) \|_2^2 ) .    
\end{aligned} 
  \label{lem 1-2} 
\end{equation} 
Meanwhile, the expansion formula~\eqref{Fourier-3} implies a similar identity for ${\cal G}_h^\frac12 D_x f$ 
\begin{equation}
  ( {\cal G}_h^\frac12 D_x f )_{i+\frac12,j,k} = \sum_{\ell, m,n=-K}^K \varphi_1^\frac12  (\frac12 \tau \lambda_{\ell,m,n} ) \cdot \frac{2}{h}  \sin (\ell \pi h) i  
  \hat{f}_{\ell,m,n} {\rm e}^{2 \pi \mathrm{i} ( \ell x_{i+\frac12} + m y_j+ n z_k)} , 
    \label{lem 1-3}
\end{equation}
with $\varphi_1 (x) = \frac{1 - {\rm e}^{-x}}{x}$ (for $x \ge 0$), and similar expansions would be available for ${\cal G}_h^\frac12 D_y f$ and ${\cal G}_h^\frac12 D_z f$. Subsequently, an application of Parseval equality reveals that 
\begin{equation} 
\begin{aligned} 
  & 
  \| {\cal G}_h^\frac12 \nabla_h  f \|_2^2 
\\
  =  & 
   \sum_{\ell, m,n=-K}^K
    \varphi_1  (\frac12 \tau \lambda_{\ell,m,n} ) \cdot \frac{4}{h^2}  ( \sin^2 (\ell \pi h) 
    + \sin^2 (m \pi h)  + \sin^2 (n \pi h) )   | \hat{f}_{\ell,m,n} |^2 
\\
   =  & 
   \sum_{\ell, m,n=-K}^K
   \frac{1 - {\rm e}^{- \frac12 \tau \lambda_{\ell,m,n}} }{\frac12 \tau \lambda_{\ell,m,n}}  
   \cdot \lambda_{\ell, m, n}   | \hat{f}_{\ell,m,n} |^2 
   \le \frac{2}{\tau} \sum_{\ell, m,n=-K}^K   | \hat{f}_{\ell,m,n} |^2 
   = \frac{2}{\tau} \| f \|_2^2 . 
\end{aligned} 
  \label{lem 1-4} 
\end{equation} 
As a result, a combination of~\eqref{lem 1-2} and \eqref{lem 1-4} leads to inequality~\eqref{lem 1-1-1}. 

Equality~\eqref{lem 1-1-2} comes from a rewritten form of the numerical algorithm~\eqref{num_scheme}
\begin{equation} 
  ( \frac{1}{\tau} + A ) ( \tilde{\phi}^{n+1} - \phi^n ) = \frac{1}{2}\mathcal{G}_h \Delta_h \phi^n 
  - V(\mathbf{x})\phi^n - \beta|\phi^n|^2\phi^n + \lambda^n \phi^n . 
  \label{lem 1-5} 
\end{equation} 
In turn, taking a discrete inner product with $\phi^n$ on both sides yields 
\begin{equation} 
\begin{aligned} 
  & 
  ( \frac{1}{\tau} + A ) \langle \tilde{\phi}^{n+1} - \phi^n , \phi^n \rangle 
\\
  = & 
    - \frac{1}{2} \| \mathcal{G}_h^\frac12 \nabla_h \phi^n \|_2^2 
  - \langle V(\mathbf{x})\phi^n + \beta|\phi^n|^2\phi^n , \phi^n \rangle 
  + \lambda^n \| \phi^n \|_2^2 = 0 , 
\end{aligned} 
  \label{lem 1-6} 
\end{equation} 
in which the summation by parts formulas~\eqref{summation by parts-1}, \eqref{Fourier-4}, representation formula~\eqref{lambda_disc} for $\lambda^n$, as well as the fact that $\| \phi^n \|_2=1$, have been used in the derivation. This in turn proves equality~\eqref{lem 1-1-2}. 

Based on the orthogonality equality~\eqref{lem 1-1-2}, we see that 
\begin{equation} 
   \| \tilde{\phi}^{n+1} \|_2^2 = \| \tilde{\phi}^{n+1} - \phi^n \|_2^2 + \| \phi^n \|_2^2   
    + 2 \langle \tilde{\phi}^{n+1} - \phi^n , \phi^n \rangle 
   = \| \tilde{\phi}^{n+1} - \phi^n \|_2^2 + \| \phi^n \|_2^2 . 
   \label{lem 1-7} 
\end{equation}  
Therefore, inequality~\eqref{lem 1-1-3} is proved. This finishes the proof of Lemma~\ref{lem 1: preliminary}. 
\end{proof} 

With the help of these preliminary estimates, a modified free energy stability becomes available.

\begin{prop}  \label{prop: energy stability} 
Assume that the artificial regularization parameter satisfies the following lower bound: 
\begin{equation} 
  A \ge \beta ( \frac12 \| \tilde{\phi}^{n+1} \|_\infty^2 + \| \phi^n \|_\infty^2 ) 
   + \frac12 \| V \|_\infty . 
  \label{condition-A-0} 
\end{equation} 
Then the numerical scheme~\eqref{num_scheme} preserves a modified free energy stability: 
\begin{equation} 
\begin{aligned} 
  & 
  E_h (\phi^{n+1} ) \le E_h (\phi^n) ,  \quad \mbox{with} 
\\
  & 
  E_h (\phi) := \frac12 \| {\cal G}_h^\frac12 \nabla_h \phi \|_2^2 + \langle V(\mathbf{x})\phi , \phi \rangle + \frac{\beta}{2} \langle |\phi|^2 \phi, \phi \rangle . 
\end{aligned} 
  \label{ener stab-0} 
\end{equation} 
\end{prop}

\begin{proof} 
Taking a discrete inner product with~\eqref{num_scheme} by $\tilde{\phi}^{n+1} - \phi^n$ gives 
\begin{equation} 
\begin{aligned} 
  & 
  ( \frac{1}{\tau} + A ) \| \tilde{\phi}^{n+1} - \phi^n \|_2^2  
  - \Big\langle \frac{1}{2}\mathcal{G}_h \Delta_h \phi^n , \tilde{\phi}^{n+1} - \phi^n \Big\rangle 
  + \langle V(\mathbf{x}) \phi^n , \tilde{\phi}^{n+1} - \phi^n \rangle 
\\
  & 
  + \beta \langle |\phi^n|^2\phi^n , \tilde{\phi}^{n+1} - \phi^n \rangle  
  - \lambda^n \langle \tilde{\phi}^{n+1} - \phi^n, \phi^n \rangle = 0 . 
\end{aligned} 
  \label{ener stab-1} 
\end{equation} 
The first two terms could be bounded with the help of the preliminary inequality~\eqref{lem 1-1-1}. The third term could be handled by a straightforward triangular equality: 
\begin{equation} 
\begin{aligned} 
  \langle V(\mathbf{x}) \phi^n , \tilde{\phi}^{n+1} - \phi^n \rangle  
  = & \frac12 \Big( \langle V({\bf x}) , | \tilde{\phi}^{n+1} |^2 \rangle 
  - \langle V({\bf x}) , | \tilde{\phi}^n |^2 \rangle  
  - \langle V({\bf x}) , | \tilde{\phi}^{n+1} - \phi^n |^2 \rangle \Big). 
\end{aligned} 
    \label{ener stab-2} 
\end{equation} 
Moreover, the nonlinear energy estimate turns out to be classical: 
\begin{equation} 
\begin{aligned} 
  \langle |\phi^n|^2\phi^n , \tilde{\phi}^{n+1} - \phi^n \rangle  
  \ge & \frac14 (  \langle | \tilde{\phi}^{n+1} |^4 , 1 \rangle - \langle | \phi^n |^4 , 1 \rangle  ) 
\\
   & 
    -  \Big\langle  \frac12 | \tilde{\phi}^{n+1} | ^2  + | \phi^n |^2 , 
    | \tilde{\phi}^{n+1} - \phi^n |^2  \Big\rangle . 
 \end{aligned} 
    \label{ener stab-3} 
\end{equation} 
Meanwhile, the last term on the left hand side of~\eqref{ener stab-1} disappears, which comes from the $L^2$ orthogonal equality~\eqref{lem 1-1-2}. In turn, a substitution of~\eqref{lem 1-1-1}, \eqref{lem 1-1-2}, \eqref{ener stab-2} and \eqref{ener stab-3} into \eqref{ener stab-1} results in 
\begin{equation} 
\begin{aligned} 
  & 
  \frac12 ( E_h  (\tilde{\phi}^{n+1}) - E_h (\phi^n) ) 
  + A \| \tilde{\phi}^{n+1} - \phi^n \|_2^2  
\\
  & 
  -  \Big\langle  \beta ( \frac12 | \tilde{\phi}^{n+1} | ^2  + | \phi^n |^2 ) + \frac12 V ({\bf x}) , 
    | \tilde{\phi}^{n+1} - \phi^n |^2  \Big\rangle \le 0 . 
\end{aligned} 
    \label{ener stab-4} 
\end{equation} 
Subsequently, under the lower bound constraint~\eqref{condition-A-0}, we see that 
\begin{equation} 
\begin{aligned} 
  & 
   A \| \tilde{\phi}^{n+1} - \phi^n \|_2^2   
  -  \Big\langle  \beta ( \frac12 | \tilde{\phi}^{n+1} | ^2  + | \phi^n |^2 ) + \frac12 V ({\bf x}) , 
    | \tilde{\phi}^{n+1} - \phi^n |^2  \Big\rangle \ge 0 , 
\\
  & 
  \mbox{so that} \quad E_h  (\tilde{\phi}^{n+1}) \le E_h (\phi^n) . 
\end{aligned} 
    \label{ener stab-5} 
\end{equation} 
On the other hand, by the fact that $\| \tilde{\phi}^{n+1} \|_2 \ge \| \phi^n \|_2 = 1$ (as given by~\eqref{lem 1-1-3}), combined with the normalization formula $\phi^{n+1} = \frac{\tilde{\phi}^{n+1}}{\|\tilde{\phi}^{n+1}\|_2}$, it is clear that 
\begin{equation} 
  E_h (\phi^{n+1}) = E_h \Big(  \frac{\tilde{\phi}^{n+1}}{\|\tilde{\phi}^{n+1}\|_2} \Big) \le E_h  (\tilde{\phi}^{n+1}) \le E_h (\phi^n) , \label{ener stab-6} 
\end{equation} 
in which the second step comes from the fact that all the energy expansion terms of $E_h$ are in either quadratic or fourth order polynomial ones. The proof of Proposition~\ref{prop: energy stability} is finished. 
\end{proof} 

Of course, the maximum bounds for the numerical solution, including both $\phi^n$ and $\tilde{\phi}^{n+1}$, are needed to theoretically justify the lower bound of $A$ in~\eqref{condition-A-0}. On the other hand, a global-in-time $\ell^\infty$ bound for $\phi^n$ and $\tilde{\phi}^{n+1}$ is not theoretically available. Instead, we have to perform a local-in-time convergence analysis and error estimate, in the $\ell^\infty (0, T; H_h^2) \cap \ell^2 (0, T; H_h^3)$ space. In turn, an application of discrete Sobolev inequality enables us to derive the desired maximum norm bound for the numerical solution, so that the modified energy stability analysis forms a closed argument. 

As a result of the modified energy estimate~\eqref{ener stab-0}, the following uniform-in-time $H_h^1$ and $\ell^4$ bounds of the numerical solution becomes valid, which will be useful in the later convergence analysis: 
\begin{equation} 
\begin{aligned} 
  & 
  E_h (\phi^k) \le E_h (\phi^0) := C_0 ,  \quad \mbox{so that} 
\\
  & 
  \| {\cal G}_h^\frac12 \nabla_h \phi^k \|_2 \le ( 2 C_0 )^\frac12 , \, \, \, 
  \| \phi^k \|_4 \le ( 2 \beta^{-1} C_0 )^\frac14 ,  \quad \forall k \ge 0. 
\end{aligned} 
  \label{ener stab-H1-1} 
\end{equation} 

\section{Convergence analysis and error estimate}  \label{sec: convergence analysis} 

Denote $\phi_e$ as the exact solution for the normalized gradient flow equation~\eqref{grad_flow_cont}. For the convenience of the normalization error estimate, we introduce the Fourier projection of the exact solution into $\mathcal{B}^{K}$, the space of trigonometric polynomials of degree to $K$ (with $N=2 K+1)$:  $\Phi_{N}(\cdot, t):=\mathcal{P}_{N} \phi_e (\cdot, t)$. Of course, a standard projection estimate is available: 
\begin{equation} \label{eqn: projection estimate}
        \| \Phi_{N} - \phi_e \|_{L^{\infty} (0, T ; H^{k} )}    
           \le Ch^{m-k} \| \phi_e \|_{L^{\infty} (0, T ; H^m )}  , 
\end{equation} 
for any $m \in \mathbb{N}$ with $0 \leq k \leq m$, $\phi_e \in L^{\infty} (0, T ; H_{\mathrm{per}}^m (\Omega) )$. In particular, a spectral accuracy of the $L^2$ norm of $\Phi_N$ is observed: 
\begin{equation} 
  \| \Phi_N (\cdot , t) \| = \| \phi_e ( \cdot , t) \|  - O (h^m) = 1 - O (h^m) . 
  \label{exact-2} 
\end{equation} 
Subsequently, we introduce an $L^2$ renormalized exact solution as 
\begin{equation} 
  \Phi (\cdot , t) := \frac{\Phi_N (\cdot , t)}{\| \Phi_N (\cdot , t) \|} . 
  \label{exact-3} 
\end{equation} 
Of course, such a renormalization leads to a constant $L^2$ norm for $\Phi$, at both the continuous and discrete levels: 
\begin{equation} 
  \| \Phi (\cdot, t) \|_2 = \| \Phi (\cdot, t) \| = \frac{\| \Phi_N (\cdot , t) \|}{\| \Phi_N (\cdot , t) \|} \equiv 1 ,  
  \quad \forall t \ge 0 ,  \label{exact-4} 
\end{equation} 
in which the first step is based on the fact that $\Phi \in \mathcal{B}^{K}$, so that its continuous and discrete $L^2$ norms are equivalent. Meanwhile, by the $L^2$ norm spectral accuracy~\eqref{exact-2} of $\Phi_N$, we see that 
\begin{equation} \label{exact-5}
        \| \Phi - \phi_e \|_{L^{\infty} (0, T ; H^{k} )}    
           \le Ch^{m-k} \| \phi_e \|_{L^{\infty} (0, T ; H^m )}  .  
\end{equation} 
To simplify the notation in the later analysis, we denote $\Phi^{m}= \Phi \left(\cdot, t_{m}\right)$ (with $t_{m}=m \cdot \tau$). In turn, the error grid function is defined as 
\begin{equation} 
  e^m := \Phi^m - \phi^m ,  \quad 
  \tilde{e}^m := \Phi^m - \tilde{\phi}^m ,  \quad \forall m \ge 1. 
  \label{error function-1} 
\end{equation} 
  The following theorem is the main result of this section.

\begin{thm}
	\label{thm:convergence}
Given initial data $\phi_e (\, \cdot \, ,t=0) \in C^6_{\rm per}(\Omega)$, suppose the exact solution for the normalized gradient flow~\eqref{grad_flow_cont} is of regularity class $\mathcal{R}$. Provided that $\tau$ and $h$ are sufficiently small, 
we have
	\begin{equation}
\| e^m \|_2 + \| \Delta_h e^m \|_2 
 + \Bigl( \tau  \sum_{k=0}^{m-1} \| {\cal G}_h^\frac12 \nabla_h \Delta_h e^k \|_2^2  ) \Bigr)^{1/2}  
   \le C ( \tau + h^2) , 
	\label{convergence-0}
	\end{equation}
for any $m \in \mathbb{N}$ (with $t_m=m\tau \le T$), and $C>0$ is independent of $\tau$ and $h$.
	\end{thm}

A careful Taylor expansion for $\Phi$ (in both time and space), combined with the projection estimates, gives the local truncation error estimate 
\begin{equation}\label{consistency-1}
\begin{aligned}  
\frac{\Phi^{n+1} - \Phi^n}{\tau} = & \frac{1}{2}\mathcal{G}_h \Delta_h \Phi^n - V(\mathbf{x}) \Phi^n - \beta|\Phi^n|^2 \Phi^n + \Lambda^n \Phi^n \\ 
 &  
  - A (\Phi^{n+1} - \Phi^n ) + \zeta_0^n , \, \, \, 
  \| \zeta_0^n \|_2 + \| \nabla_h \zeta_0^n \|_2 \le C (\tau + h^2 ) , \\ 
 \Lambda^n = & \frac12 \| {\cal G}_h^\frac12 \nabla_h \Phi^n \|_2^2 + \langle V(\mathbf{x})\Phi^n + \beta|\Phi^n|^2 \Phi^n, \Phi^n \rangle ,  \\ 
 \| \Phi^{n+1} \|_2 = & \| \Phi^n \|_2 = 1 . 
\end{aligned} 
\end{equation} 
Subsequently, a subtraction of the numerical system~\eqref{num_scheme} from the consistency estimate \eqref{consistency-1} yields 
\begin{equation} 
\begin{aligned}  
\frac{\tilde{e}^{n+1} - e^n}{\tau} = & \frac{1}{2}\mathcal{G}_h \Delta_h e^n - V(\mathbf{x}) e^n 
 - \beta {\cal NLE}^n + \lambda^n e^n + e_\lambda^n \Phi^n \\ 
 &  
  - A (e^{n+1} - e^n ) + \zeta_0^n ,  \quad \mbox{with} \, \, \, 
  \| \zeta_0^n \|_2 , \, \| \Delta_h \zeta_0^n \|_2 \le C (\tau + h^2) ,  \\ 
  {\cal NLE}^n := & |\phi^n|^2 e^n + ( (\Phi^n)^c e^n + \phi^n ( e^n )^c ) \Phi^n , 
\\
 e_\lambda^n = & \frac12 \langle {\cal G}_h^\frac12 \nabla_h ( \Phi^n + \phi^n ) , 
  {\cal G}_h^\frac12 \nabla_h e^n \rangle  
  + \langle V(\mathbf{x}) ( \Phi^n + \phi^n ) , e^n \rangle 
\\
  & 
  + \beta ( \langle |\phi^n|^2 \phi^n, e^n \rangle 
 + \langle {\cal NLE}^n , \Phi^n \rangle ) ,  \\ 
 e^{n+1}  = & \tilde{e}^{n+1} + ( \| \tilde{\phi}^{n+1} \|_2 - 1) \phi^{n+1}  . 
\end{aligned} 
  \label{error equation-1}
\end{equation} 

A discrete $\ell^\infty$ and $H_h^2$ bound is assumed for the constructed approximate solution $\Phi$, since its functional bound only depends on the exact solution: 
\begin{equation} 
  \| \Phi^k \|_\infty , \, \,  \| \Phi^k \|_2 , \, \,  \| \Delta_h \Phi^k \|_2 \le C^* .  
  \label{exact-inf-1} 
\end{equation} 
Of course, in terms of $V ({\bf x})$ in the potential energy, a natural $\ell^\infty \cap H_h^2$ bound, $\| V \|_\infty\le C^*$, $\| V \|_2$, $\| \Delta_h V \|_2 \le C^*$, could be assumed. In addition, we make the following a-priori assumption for the numerical error function at the previous time step:
\begin{equation}\label{a priori-1}
\| e^n \|_2  , \,  \| \Delta_h e^n \|_2 \le \tau^{\frac{7}{8}}+h^{\frac{7}{4}} . 
\end{equation}
Such an a-priori assumption is valid at $n=0$, and it will be recovered by the optimal rate convergence analysis at the next time step, as will be proved later. In turn, a discrete $\ell^\infty$ bound for the numerical error function is available at the previous time step, based on the discrete Sobolev inequalities in~\eqref{prop 1-1-2}:  
\begin{equation}\label{a priori-2}
 \| e^n \|_\infty \leq \breve{C}_0 ( \| e^n \|_2 + \| \Delta_h e^n \|_2 )  
  \le C(\tau^\frac78+h^{\frac74})  \le \frac12 . 
\end{equation} 
Subsequently, combined with the regularity assumption~\eqref{exact-inf-1}, an $\ell^\infty \cap H_h^2$ bound for the numerical solution could be derived at the previous time step: 
\begin{equation}\label{a priori-3}
\begin{aligned}
&\|  \phi^n \|_\infty \le \| \Phi^{k}\|_\infty + \| e^n \|_\infty \leq C^*+\frac12:= \tilde{C}_1 , 
  \\ 
  & 
  \|  \Delta_h \phi^n \|_2 \le \| \Delta_h \Phi^n \|_2 + \| \Delta_h e^n \|_2 
  \le C^* + \tau^\frac78 + h^\frac74 \le C^*+\frac12 = \tilde{C}_1 . 
\end{aligned}
\end{equation} 

The following preliminary estimates will be needed in the later convergence analysis; the detailed proof will be provided in Appendix~\ref{appendix: prop 2}. 

\begin{prop} \label{prop: convergence-prelim} 
We have 
\begin{align} 
  & 
  \| f \|_2^2 \ge \frac{\tau}{2} \| {\cal G}_h^\frac12 \nabla_h f \|_2^2 , \quad 
  \| \Delta_h f \|_2^2 \ge \frac{\tau}{2} \| {\cal G}_h^\frac12 \nabla_h \Delta_h f \|_2^2 , 
   \quad \forall f \in {\mathcal C}_{\rm per} ,  \label{prop 1-1-1} 
\\
  & 
     \| \nabla_h f \|_4 \le \breve{C}_0 \| \Delta_h f \|_2 , \quad    
   \| f \|_\infty , \, \| f \|_{H_h^2} \le \breve{C}_0 ( \| f \|_2 + \| \Delta_h f \|_2 ) ,  
   \quad  \forall f \in {\mathcal C}_{\rm per} ,  
   \label{prop 1-1-2}  
\\
  & 
   \| \Delta_h ( f g ) \|_2 \le \breve{C}_1 ( \| f \|_2 + \| \Delta_h f \|_2 ) ( \| g \|_2 + \| \Delta_h g \|_2 ) ,  
   \quad \forall f , g \in {\mathcal C}_{\rm per} ,  
   \label{prop 1-1-2-2}  
\\
  & 
  \| {\cal NLE}^n \|_2 \le \tilde{C}_2 \| e^n \|_2 ,  \quad 
  \| \Delta_h {\cal NLE}^n \|_2 \le \tilde{C}_3 ( \| e^n \|_2 + \| \Delta_h e^n \|_2 ) , 
  \label{prop 1-1-3}  
\\
  & 
  0 \le \lambda^n \le \tilde{C}_4 ,  \quad 
  | e_\lambda^n | \le \tilde{C}_5 ( \| e^n \|_2 + \| {\cal G}_h^\frac12 \nabla_h e^n \|_2 ) , 
  \label{prop 1-1-4}  
\end{align} 
in which $\breve{C}_0$ and $\breve{C}_1$ depend only on $\Omega$, and $\tilde{C}_j$, $2 \le j \le 5$, depends only on the exact solution. 
\end{prop} 

In terms of the discrete $\ell^2$ renormalization process, the following comparison estimates between $\tilde{e}^{n+1}$ and $e^{n+1}$ will play an important role in the convergence analysis. The detailed proof will be presented in Appendix~\ref{appendix: prop 3}.

\begin{prop} \label{prop: renormalization} 
Assume that $\| \tilde{\phi}^{n+1} \|_2 \ge 1$. Under an a-priori assumption that 
\begin{equation}\label{a priori-4}
   \| \Delta_h \tilde{e}^{n+1} \|_2 \le 2 ( \tau^{\frac{7}{8}}+h^{\frac{7}{4}} ),  
\end{equation} 
we have the following estimates 
\begin{align} 
  & 
  0 \le \| \tilde{\phi}^{n+1} \|_2 - 1 \le \tilde{C}_6 \tau^2 ,  \label{prop 2-1-0} 
\\
  & 
  \| e^{n+1} \|_2^2 + ( \| \tilde{\phi}^{n+1} \|_2 - 1)^2 
  \le \| \tilde{e}^{n+1} \|_2^2 \le 2 ( \| e^{n+1} \|_2^2 + ( \| \tilde{\phi}^{n+1} \|_2 - 1)^2 ) , 
  \label{prop 2-1-1} 
\\
  & 
  \| \Delta_h e^{n+1} \|_2^2 \le ( 1 +  \tau ) \| \Delta_h \tilde{e}^{n+1} \|_2^2 
  + \tilde{C}_7 \tau^{-1} ( \| \tilde{\phi}^{n+1} \|_2 - 1)^2 , 
  \label{prop 2-1-2} 
\\
  & 
  \| \Delta_h \tilde{e}^{n+1} \|_2^2 \le 2 \| \Delta_h e^{n+1} \|_2^2 
  + \tilde{C}_8 ( \| \tilde{\phi}^{n+1} \|_2 - 1)^2 , 
  \label{prop 2-1-3} 
\end{align} 
in which $\tilde{C}_j$,  $6 \le j \le 8$, depends only on the exact solution. 
\end{prop}

\subsection{A preliminary $\ell^\infty (0, T; \ell^2) \cap \ell^2 (0, T; H_h^1)$ error estimate at the intermediate stage} 
  
Taking a discrete inner product with~\eqref{error equation-1} by $2 \tilde{e}^{n+1}$ gives  
\begin{equation} 
\begin{aligned}  
  & 
( \frac{1}{\tau} + A ) ( \| \tilde{e}^{n+1} \|_2^2 - \| e^n \|_2^2 +  \| \tilde{e}^{n+1} - e^n \|_2^2 ) 
  +   \langle \mathcal{G}_h^\frac12 \nabla_h e^n , 
   \mathcal{G}_h^\frac12 \nabla_h \tilde{e}^{n+1} \rangle 
\\
  = &   
   - 2 \langle V(\mathbf{x}) e^n , \tilde{e}^{n+1} \rangle 
 - 2 \beta \langle {\cal NLE}^n , \tilde{e}^{n+1} \rangle  
  + 2 \lambda^n \langle e^n , \tilde{e}^{n+1} \rangle  
\\
  & 
  + 2 e_\lambda^n \langle \Phi^n , \tilde{e}^{n+1} \rangle 
  + 2 \langle \zeta_0^n ,  \tilde{e}^{n+1} \rangle ,  
\end{aligned} 
  \label{convergence-L2-1}
\end{equation} 
with an application of the summation by parts formula. The numerical error diffusion inner product on the left hand side could be analyzed as follows 
\begin{equation} 
\begin{aligned} 
  & 
  \langle \mathcal{G}_h^\frac12 \nabla_h e^n , 
   \mathcal{G}_h^\frac12 \nabla_h \tilde{e}^{n+1} \rangle  
\\
  = & 
   \frac12 ( \| \mathcal{G}_h^\frac12 \nabla_h e^n \|_2^2  
   + \| \mathcal{G}_h^\frac12 \nabla_h \tilde{e}^{n+1} \|_2^2  
   - \| \mathcal{G}_h^\frac12 \nabla_h ( \tilde{e}^{n+1} - e^n ) \|_2^2 )  
\\
  \ge & 
   \frac12 ( \| \mathcal{G}_h^\frac12 \nabla_h e^n \|_2^2  
   + \| \mathcal{G}_h^\frac12 \nabla_h \tilde{e}^{n+1} \|_2^2  ) 
   - \frac{1}{\tau} \| \tilde{e}^{n+1} - e^n  \|_2^2 , 
\end{aligned} 
  \label{convergence-L2-2}
\end{equation} 
in which the preliminary inequality~\eqref{prop 1-1-1} in Proposition~\ref{prop: convergence-prelim} has been applied in the last step. Furthermore, the right hand side terms could be bounded with the help of inequalities~\eqref{prop 1-1-3}, \eqref{prop 1-1-4} in Proposition~\ref{prop: convergence-prelim}:  
\begin{align} 
  & 
  - 2 \langle V(\mathbf{x}) e^n , \tilde{e}^{n+1} \rangle 
  \le 2 \| V \|_\infty \cdot \| e^n \|_2 \cdot \| \tilde{e}^{n+1} \|_2 
  \le  2 C^* \| e^n \|_2 \cdot \| \tilde{e}^{n+1} \|_2 , 
   \label{convergence-L2-3-1} 
\\
  & 
  - 2 \beta \langle {\cal NLE}^n , \tilde{e}^{n+1} \rangle   
  \le 2 \beta \| {\cal NLE}^n \|_2 \cdot \| \tilde{e}^{n+1} \|_2 
  \le 2 \tilde{C}_2 \beta \| e^n \|_2 \cdot \| \tilde{e}^{n+1} \|_2  ,  
  \label{convergence-L2-3-2} 
\\
  & 
  2 \lambda^n \langle e^n , \tilde{e}^{n+1} \rangle 
  \le 2 \lambda^n \| e^n \|_2 \cdot \| \tilde{e}^{n+1} \|_2 
  \le 2 \tilde{C}_4 \| e^n \|_2 \cdot \| \tilde{e}^{n+1} \|_2 , 
  \label{convergence-L2-3-3} 
\\
  & 
  2 e_\lambda^n \langle \Phi^n , \tilde{e}^{n+1} \rangle 
  \le 2 | e_\lambda^n | \cdot \| \Phi^n \|_2 \cdot \| \tilde{e}^{n+1} \|_2 
  \le  2 \tilde{C}_5 ( \| e^n \|_2 + \| {\cal G}_h^\frac12 \nabla_h e^n \|_2 ) \| \tilde{e}^{n+1} \|_2  
  \nonumber 
\\
  & \qquad 
  \le 2 \tilde{C}_5 \| e^n \|_2 \cdot \| \tilde{e}^{n+1} \|_2 
  + \frac14 \| {\cal G}_h^\frac12 \nabla_h e^n \|_2^2 
  + 4 \tilde{C}_5^2 \| \tilde{e}^{n+1} \|_2^2 , 
  \label{convergence-L2-3-4} 
\\
  & 
  2 \langle \zeta_0^n ,  \tilde{e}^{n+1} \rangle 
  \le 2 \| \zeta_0^n \|_2 \cdot \|  \tilde{e}^{n+1} \|_2 
  \le \| \zeta_0^n \|_2^2 + \|  \tilde{e}^{n+1} \|_2^2 . 
  \label{convergence-L2-3-5} 
\end{align} 
Subsequently, a substitution of~\eqref{convergence-L2-2}-\eqref{convergence-L2-3-5} into \eqref{convergence-L2-1} leads to 
\begin{equation} 
\begin{aligned}  
  & 
( \frac{1}{\tau} + A ) ( \| \tilde{e}^{n+1} \|_2^2 - \| e^n \|_2^2 ) 
  +   \frac14 \| \mathcal{G}_h^\frac12 \nabla_h e^n \|_2^2  
   + \frac12 \| \mathcal{G}_h^\frac12 \nabla_h \tilde{e}^{n+1} \|_2^2  
\\
  \le &   
    \tilde{C}_9 ( \| e^n \|_2^2 +\| \tilde{e}^{n+1} \|_2^2 )   
    + ( 4 \tilde{C}_5^2 +1 ) \| \tilde{e}^{n+1} \|_2^2 + \| \zeta_0^n \|_2^2 ,  
\end{aligned} 
  \label{convergence-L2-4}
\end{equation} 
with $\tilde{C}_9 = C^* + \tilde{C}_2 \beta + \tilde{C}_4 + \tilde{C}_5$. 


\subsection{A preliminary $\ell^\infty (0, T; H_h^2) \cap \ell^2 (0, T; H_h^3)$ error estimate at the intermediate stage} 
  
Taking a discrete inner product with~\eqref{error equation-1} by $2 \Delta_h^2 \tilde{e}^{n+1}$ leads to  
\begin{equation} 
\begin{aligned}  
  & 
( \frac{1}{\tau} + A ) ( \| \Delta_h \tilde{e}^{n+1} \|_2^2 - \| \Delta_h e^n \|_2^2 
 +  \| \Delta_h ( \tilde{e}^{n+1} - e^n ) \|_2^2 )  
\\
  & 
  +   \langle \mathcal{G}_h^\frac12 \nabla_h \Delta_h e^n , 
   \mathcal{G}_h^\frac12 \nabla_h \Delta_h \tilde{e}^{n+1} \rangle 
\\
  = &   
   - 2 \langle \Delta_h ( V e^n ) , \Delta_h \tilde{e}^{n+1} \rangle 
 - 2 \beta \langle \Delta_h {\cal NLE}^n , \Delta_h \tilde{e}^{n+1} \rangle  
\\
  & 
  + 2 \lambda^n \langle \Delta_h e^n , \Delta_h \tilde{e}^{n+1} \rangle  
  + 2 e_\lambda^n \langle \Delta_h \Phi^n , \Delta_h \tilde{e}^{n+1} \rangle 
  + 2 \langle \Delta_h \zeta_0^n ,  \Delta_h \tilde{e}^{n+1} \rangle ,  
\end{aligned} 
  \label{convergence-H2-1}
\end{equation} 
with a repeated application of the summation by parts formula. The numerical error diffusion inner product term turns out to be 
\begin{equation} 
\begin{aligned} 
  & 
  \langle \mathcal{G}_h^\frac12 \nabla_h \Delta_h e^n , 
   \mathcal{G}_h^\frac12 \nabla_h \Delta_h \tilde{e}^{n+1} \rangle  
\\
  = & 
   \frac12 ( \| \mathcal{G}_h^\frac12 \nabla_h \Delta_h e^n \|_2^2  
   + \| \mathcal{G}_h^\frac12 \nabla_h \Delta_h \tilde{e}^{n+1} \|_2^2  
   - \| \mathcal{G}_h^\frac12 \nabla_h \Delta_h ( \tilde{e}^{n+1} - e^n ) \|_2^2 )  
\\
  \ge & 
   \frac12 ( \| \mathcal{G}_h^\frac12 \nabla_h \Delta_h e^n \|_2^2  
   + \| \mathcal{G}_h^\frac12 \nabla_h \Delta_h \tilde{e}^{n+1} \|_2^2  ) 
   - \frac{1}{\tau} \| \Delta_h ( \tilde{e}^{n+1} - e^n )  \|_2^2 . 
\end{aligned} 
  \label{convergence-H2-2}
\end{equation} 
Similarly, the bounds for the right hand side terms of~\eqref{convergence-H2-1} could be derived as follows, with the help of inequalities~\eqref{prop 1-1-2-2}-\eqref{prop 1-1-4} in Proposition~\ref{prop: convergence-prelim}:  
\begin{align} 
  & 
  \| \Delta_h ( V e^n ) \|_2  \le 
  \breve{C}_1 ( \| V \|_2 + \| \Delta_h V \|_2 ) ( \| e^n \|_2 + \| \Delta_h e^n \|_2 )  \nonumber 
\\
  & \qquad 
  \le 2 \breve{C}_1 C^*  ( \| e^n \|_2 + \| \Delta_h e^n \|_2 ) ,  
   \label{convergence-H2-3-0}  
\\
  & 
  - 2 \langle \Delta_h ( V e^n ) , \Delta_h \tilde{e}^{n+1} \rangle 
  \le 2 \| \Delta_h ( V e^n ) \|_2 \cdot \| \Delta_h \tilde{e}^{n+1} \|_2  \nonumber 
\\
  & \qquad 
  \le 4  \breve{C}_1 C^* ( \| e^n \|_2 + \| \Delta_h e^n \|_2 )  
   \| \Delta_h \tilde{e}^{n+1} \|_2 , 
   \label{convergence-H2-3-1} 
\\
  & 
  - 2 \beta \langle \Delta_h {\cal NLE}^n , \Delta_h \tilde{e}^{n+1} \rangle   
  \le 2 \beta  \| \Delta_h {\cal NLE}^n \|_2 \cdot \| \Delta_h \tilde{e}^{n+1} \|_2  \nonumber 
\\
  & \qquad 
  \le 2 \tilde{C}_3 \beta ( \| e^n \|_2 + \| \Delta_h e^n \|_2 ) \| \Delta_h \tilde{e}^{n+1} \|_2  ,  
  \label{convergence-H2-3-2} 
\\
  & 
  2 \lambda^n \langle \Delta_h e^n , \Delta_h \tilde{e}^{n+1} \rangle 
  \le 2 \lambda^n \| \Delta_h e^n \|_2 \cdot \| \Delta_h \tilde{e}^{n+1} \|_2 
  \le 2 \tilde{C}_4 \| \Delta_h e^n \|_2 \cdot \| \Delta_h \tilde{e}^{n+1} \|_2 , 
  \label{convergence-H2-3-3} 
\\
  & 
  2 e_\lambda^n \langle \Delta_h \Phi^n , \Delta_h \tilde{e}^{n+1} \rangle 
  \le 2 | e_\lambda^n | \cdot \| \Delta_h \Phi^n \|_2 \cdot \| \Delta_h \tilde{e}^{n+1} \|_2  
  \nonumber 
\\
  & \qquad 
  \le  2 \tilde{C}_5 C^* ( \| e^n \|_2 + \| {\cal G}_h^\frac12 \nabla_h e^n \|_2 ) 
   \| \nabla_h \tilde{e}^{n+1} \|_2   \nonumber 
\\
  & \qquad 
  \le  2 \tilde{C}_5 C^* \breve{C}_0 | \Omega|^\frac14 ( \| e^n \|_2 + \| \Delta_h e^n \|_2 ) 
   \| \Delta_h \tilde{e}^{n+1} \|_2 , 
  \label{convergence-H2-3-4} 
\\
  & 
  2 \langle \Delta_h \zeta_0^n ,  \Delta_h \tilde{e}^{n+1} \rangle 
  \le 2 \| \Delta_h \zeta_0^n \|_2 \cdot \| \Delta_h \tilde{e}^{n+1} \|_2 
  \le \| \Delta_h \zeta_0^n \|_2^2 + \|  \Delta_h \tilde{e}^{n+1} \|_2^2 . 
  \label{convergence-H2-3-5} 
\end{align} 
In particular, we notice that inequality $\| {\cal G}_h^\frac12 \nabla_h e^n \|_2 \le \| \nabla_h e^n \|_2 \le \breve{C}_0 | \Omega |^\frac14 \| \Delta_h e^n \|_2$ (which comes from~\eqref{prop 1-1-2}), has been used in the derivation of~\eqref{convergence-H2-3-4}. In turn, a substitution of~\eqref{convergence-H2-2}-\eqref{convergence-H2-3-5} into \eqref{convergence-H2-1} leads to 
\begin{equation} 
\begin{aligned}  
  & 
( \frac{1}{\tau} + A ) ( \| \Delta_h \tilde{e}^{n+1} \|_2^2 - \| \Delta_h e^n \|_2^2 ) 
  +   \frac12 ( \| \mathcal{G}_h^\frac12 \nabla_h \Delta_h e^n \|_2^2  
   + \| \mathcal{G}_h^\frac12 \nabla_h \Delta_h \tilde{e}^{n+1} \|_2^2 ) 
\\
  \le &   
    \tilde{C}_{10} ( 2 \| e^n \|_2^2 + 2 \| \Delta_h e^n \|_2^2 + \| \Delta_h \tilde{e}^{n+1} \|_2^2 )   
    +  \| \Delta_h \tilde{e}^{n+1} \|_2^2 + \| \Delta_h \zeta_0^n \|_2^2 ,  
\end{aligned} 
  \label{convergence-H2-4}
\end{equation} 
with $\tilde{C}_{10} = 2 \breve{C}_1 C^* + \tilde{C}_3 \beta + \tilde{C}_4 + \tilde{C}_5 C^* \breve{C}_0 | \Omega|^\frac14$. As a direct consequence of the a-priori assumption~\eqref{a priori-1} and the truncation error estimate $\| \Delta_h \zeta_0^n \|_2 \le C (\tau + h^2)$, it is straightforward to verify that 
\begin{equation} 
  \| \Delta_h \tilde{e}^{n+1} \|_2 \le 2 ( \tau^\frac78 + h^\frac74) . \label{a priori-5} 
\end{equation} 
Of course, either~\eqref{convergence-H2-4} or \eqref{a priori-5} could not be used as an induction-style argument. Such a preliminary error estimate plays a role of a $H_h^2$ rough error estimate, which would enable us to derive a refined renormalization estimate in Proposition~\ref{prop: renormalization}. 


\subsection{A refined error estimate} 

As mentioned above, the preliminary error estimates \eqref{convergence-L2-4} and \eqref{convergence-H2-4} are not sufficient to conduct a theoretical analysis between two consecutive time steps. Instead, we need a refined error estimate for $e^{n+1}$ to close the argument. On the other hand, the preliminary estimate~\eqref{lem 1-1-3} and the rough error estimate \eqref{a priori-5} enable us to apply Proposition~\ref{prop: renormalization}, so that the inequalities \eqref{prop 2-1-1}-\eqref{prop 2-1-3} become valid, which will play an important role in the refined error estimate. 

For instance, a combination of \eqref{convergence-L2-4} and \eqref{prop 2-1-1} leads to a unified $\ell^2$ error estimate 
\begin{equation} 
\begin{aligned}  
  & 
( \frac{1}{\tau} + A ) ( \| e^{n+1} \|_2^2 - \| e^n \|_2^2 ) 
  +   \frac14 \| \mathcal{G}_h^\frac12 \nabla_h e^n \|_2^2  
   + \frac12 \| \mathcal{G}_h^\frac12 \nabla_h \tilde{e}^{n+1} \|_2^2  
\\
  \le &   
    \tilde{C}_9 ( \| e^n \|_2^2 +\| \tilde{e}^{n+1} \|_2^2 )   
    + ( 4 \tilde{C}_5^2 +1 ) \| \tilde{e}^{n+1} \|_2^2 + \| \zeta_0^n \|_2^2  
    - ( \frac{1}{\tau} + A) ( \| \tilde{\phi}^{n+1} \|_2 - 1)^2 .   
\end{aligned} 
  \label{convergence-L2-5}
\end{equation} 

In terms of a unified $H_h^2$ error estimate, we begin with the following bound, which comes from inequalities \eqref{prop 2-1-0} and \eqref{prop 2-1-2}:  
\begin{equation} 
  \| \Delta_h \tilde{e}^{n+1} \|_2^2 \ge ( 1 - \tau ) \| \Delta_h e^{n+1} \|_2^2 
  - \tilde{C}_7 \tau^{-1} ( \| \tilde{\phi}^{n+1} \|_2 - 1)^2 
  \ge ( 1 - \tau ) \| \Delta_h e^{n+1} \|_2^2 
  - \tilde{C}_6^2 \tilde{C}_7 \tau^3 .  \label{convergence-H2-5} 
\end{equation} 
In turn, its substitution into \eqref{convergence-H2-4} gives 
\begin{equation} 
\begin{aligned}  
  & 
( \frac{1}{\tau} + A ) ( \| \Delta_h e^{n+1} \|_2^2 - \| \Delta_h e^n \|_2^2 ) 
  +   \frac12 ( \| \mathcal{G}_h^\frac12 \nabla_h \Delta_h e^n \|_2^2  
   + \| \mathcal{G}_h^\frac12 \nabla_h \Delta_h \tilde{e}^{n+1} \|_2^2 ) 
\\
  \le &   
    \tilde{C}_{10} ( 2 \| e^n \|_2^2 + 2 \| \Delta_h e^n \|_2^2 + \| \Delta_h \tilde{e}^{n+1} \|_2^2 )   
    +  \| \Delta_h \tilde{e}^{n+1} \|_2^2 + \| \Delta_h e^{n+1} \|_2^2  
\\
  & 
    + \tilde{C}_6^2 \tilde{C}_7 \tau^2 ( 1 + A \tau) 
    + \| \Delta_h \zeta_0^n \|_2^2 .  
\end{aligned} 
  \label{convergence-H2-6}
\end{equation} 
Subsequently, a combination of \eqref{convergence-H2-6} with \eqref{convergence-L2-5} results in 
\begin{equation} 
\begin{aligned}  
  & 
( \frac{1}{\tau} + A ) ( \| e^{n+1} \|_2^2 - \|  e^n \|_2^2
 + \| \Delta_h e^{n+1} \|_2^2 - \| \Delta_h e^n \|_2^2 ) 
\\
  & 
  +   \frac14 ( \| \mathcal{G}_h^\frac12 \nabla_h e^n \|_2^2   
  +   \| \mathcal{G}_h^\frac12 \nabla_h \Delta_h e^n \|_2^2 ) 
\\
  \le &   
    (\tilde{C}_9 + 2 \tilde{C}_{10} ) \| e^n \|_2^2 
    + (\tilde{C}_9 + 4 \tilde{C}_5^2 +1 ) \| \tilde{e}^{n+1} \|_2^2 
    + 2 \tilde{C}_{10} \| \Delta_h e^n \|_2^2 
\\
  & 
    + ( \tilde{C}_{10} + 1 ) \| \Delta_h \tilde{e}^{n+1} \|_2^2   
    +  \| \Delta_h e^{n+1} \|_2^2  
    + \| \zeta_0^n \|_2^2 + \| \Delta_h \zeta_0^n \|_2^2 
    + 2 \tilde{C}_6^2 \tilde{C}_7 \tau^2    
\\
  \le &   
    ( \tilde{C}_9 + 2 \tilde{C}_{10}) \| e^n \|_2^2 
    + 2 \tilde{C}_{11} ( \| e^{n+1} \|_2^2 + ( \| \tilde{\phi}^{n+1} \|_2 - 1)^2 ) 
    + 2 \tilde{C}_{10} \| \Delta_h e^n \|_2^2 
\\
  & 
    +  \| \Delta_h e^{n+1} \|_2^2  
    + ( \tilde{C}_{10} + 1 )  ( 2 \| \Delta_h e^{n+1} \|_2^2  
    + \tilde{C}_8 ( \| \tilde{\phi}^{n+1} \|_2 - 1)^2 ) ) 
\\
  & 
    + \| \zeta_0^n \|_2^2 + \| \Delta_h \zeta_0^n \|_2^2  
    + 2 \tilde{C}_6^2 \tilde{C}_7 \tau^2    
\\
  \le &   
    ( \tilde{C}_9 + 2 \tilde{C}_{10}) \| e^n \|_2^2 
    + 2 \tilde{C}_{11}  \| e^{n+1} \|_2^2 
    + 2 \tilde{C}_{10}  \| \Delta_h e^n \|_2^2 
    + \tilde{C}_{12} \| \Delta_h e^{n+1} \|_2^2  
\\
  & 
    + \| \zeta_0^n \|_2^2 + \| \Delta_h \zeta_0^n \|_2^2  
    + 2 \tilde{C}_6^2 \tilde{C}_7 \tau^2  
    + ( 2 \tilde{C}_{11} + \tilde{C}_8 ( \tilde{C}_{10} +1) ) \tilde{C}_6^2 \tau^4 
\\ 
  \le &   
    ( \tilde{C}_9 + 2 \tilde{C}_{10}) \| e^n \|_2^2 
    + 2 \tilde{C}_{11}  \| e^{n+1} \|_2^2 
    + 2 \tilde{C}_{10}  \| \Delta_h e^n \|_2^2 
    + \tilde{C}_{12} \| \Delta_h e^{n+1} \|_2^2  
\\
  & 
    + \| \zeta_0^n \|_2^2 + \| \Delta_h \zeta_0^n \|_2^2  
    + ( 2 \tilde{C}_6^2 \tilde{C}_7 + 1) \tau^2 ,  
\end{aligned} 
  \label{convergence-H2-7}
\end{equation} 
with $\tilde{C}_{11} = \tilde{C}_9 + 4 \tilde{C}_5^2 +1$, $\tilde{C}_{12} = 1 + 2 (\tilde{C}_{10} +1)$, provided that $\tau$ is sufficiently small. Notice that the preliminary renormalization estimates \eqref{prop 2-1-1} and \eqref{prop 2-1-3} have been applied in the second step in the derivation. Subsequently, we denote a quantity 
\begin{equation} 
   Q_h^k := \| e^k \|_2^2 + \| \Delta_h e^k \|_2^2 , \label{error quantity-1} 
\end{equation} 
so that the following estimate becomes available: 
\begin{equation} 
\begin{aligned}  
  & 
( \frac{1}{\tau} + A ) ( Q_h^{n+1} - Q_h^n )  
  +   \frac14 ( \| \mathcal{G}_h^\frac12 \nabla_h e^n \|_2^2   
  + \gamma_0 \tau \| \mathcal{G}_h^\frac12 \nabla_h \Delta_h e^n \|_2^2 ) 
\\
  \le &   
    ( \tilde{C}_9 + 2 \tilde{C}_{10} ) Q_h^n 
    + ( 2 \tilde{C}_{11}  +  \tilde{C}_{12} )  Q_h^{n+1}   
    + \| \zeta_0^n \|_2^2 + \| \Delta_h \zeta_0^n \|_2^2  
    + ( 2 \tilde{C}_6^2 \tilde{C}_7 + 1) \tau^2 .   
\end{aligned} 
  \label{convergence-H2-8}
\end{equation} 
Consequently, an application of the discrete Gronwall inequality gives the desired convergence estimate:
\begin{equation}
\begin{aligned}
  &
  Q_h^{n+1} \le C (\tau^2 + h^4) , \quad
  \| e^{n+1} \|_2 +  \| \Delta_h e^{n+1} \|_2 
  \le C ( Q_h^{n+1} )^\frac12 \le C (\tau + h^2) ,
\\
  &
  \tau \sum_{k=1}^n \| \mathcal{G}_h^\frac12 \nabla_h e^k \|_2^2   
  + \tau \sum_{k=1}^n \| \mathcal{G}_h^\frac12 \nabla_h \Delta_h e^k \|_2^2 
  \le C ( \tau^2 + h^4 ) ,  
\end{aligned}
  \label{convergence-H2-9}
\end{equation}
in which the truncation error accuracy $\| \zeta_0^n \|_2 , \, \| \Delta_h \zeta_0^n \|_2 \leq C(\tau +h^2)$ has been applied. This finishes the unified $H_h^2$ error estimate.

Finally, with the help of the error estimate at the next time step, we see that the a-priori assumption \eqref{a priori-1} is satisfied at $t^{n+1}$:
\begin{equation} 
 \| e^{n+1} \|_2  , \,  \| \Delta_h e^{n+1} \|_2
 \le C(\tau +h^2)  \le \tau^\frac78 + h^\frac74 , 
 \label{a priori-6} 
\end{equation}
provided that $\tau$ and $h$ are sufficiently small. The proof of Theorem~\ref{thm:convergence} is completed.

\begin{remark} 
It is observed that, the unified $\ell^2$ error estimate~\eqref{convergence-L2-5}, combined with the $\ell^2$ renormalization estimate~\eqref{prop 2-1-1}, is sufficient to derive an $\ell^\infty (0, T; \ell^2) \cap \ell^2 (0, T; H_h^1)$ error bound for the numerical solution. However, such an error estimate is not able to form a closed theoretical argument. In more details, the $\ell^\infty$ bound~\eqref{a priori-3} for the numerical solution at the previous step has played an important role in the $\ell^2$ error estimate, while such a maximum norm bound comes from the $H_h^2$ convergence estimate~\eqref{a priori-1} and the discrete Sobolev inequality~\eqref{a priori-2}. As a result, an $H_h^2$ error estimate is necessary to go through the the convergence analysis, and an application of the inverse inequality is avoided in the theoretical derivation, because of the discrete Sobolev embedding~\eqref{prop 1-1-2}, from $H_h^2$ into $\ell^\infty$ and $W_h^{1,4}$. A similar methodology of an $H_h^2$ error estimate has been reported in~\cite{gottlieb12b} to deal with the convergence analysis for the 3-D viscous Burgers' equation.
\end{remark} 

\begin{remark} 
It is noticed that, the $\ell^2$ renormalization error estimate~\eqref{prop 2-1-1} provides a detailed analysis between the $\| \cdot \|_2$ bounds of $\tilde{e}^{n+1}$, $e^{n+1}$ and $( \| \tilde{\phi}^{n+1} \|_2 -1) \phi^{n+1} = e^{n+1} - \tilde{e}^{n+1}$, which forms a functional triangle. Moreover, a careful application of Law-of-Cosine-style analysis leads to the left inequality in \eqref{prop 2-1-1}, 
and its derivation is based on a subtle fact that $\| \Phi^{n+1} \|_2 = \| \phi^{n+1} \|_2 =1$, so that the functional angle between $e^{n+1} - \tilde{e}^{n+1}$ and $e^{n+1}$ could be carefully analyzed. On the other hand, in the $\ell^\infty (0, T; H_h^2) \cap \ell^2 (0, T; H_h^3)$ error estimate, there is no constraint that $\| \Delta_h \Phi^{n+1} \|_2 = \| \Delta_h \phi^{n+1} \|_2$, so that a bound for the inner product between $\Delta_h (e^{n+1} - \tilde{e}^{n+1})$ and $\Delta_h e^{n+1}$ becomes very difficult to obtain. Instead, we have to use the Cauchy inequality to derive this bound, while an increment factor of $\tau$ in front of $\| \Delta_h e^{n+1} \|_2^2$ gives a singular $\tau^{-1}$ factor for $( \| \tilde{\phi}^{n+1} \|_2 -1)^2$. This derivation leads to the renormalized $H_h^2$ error estimate~\eqref{prop 2-1-2}. 

To overcome the difficulty associated with this singular $\tau^{-1}$ coefficient, we make another key observation that, $\| \tilde{\phi}^{n+1} \|_2 -1$ is of order $O (\tau^2)$ (instead of $O (\tau)$, as stated in~\eqref{prop 2-1-0}. Such an $O (\tau^2)$ accuracy order comes from the rewritten form~\eqref{lem 1-5} of the numerical solution, $H_h^2$ convergence estimate at the previous time step, combined with an $\ell^2$ pythagorean identity (due to the orthogonality~\eqref{lem 1-1-2} between $\phi^n$ and $\tilde{\phi}^{n+1} - \phi^n$). With the help of $O (\tau^2)$ estimate~\eqref{prop 2-1-0} for $\| \tilde{\phi}^{n+1} \|_2 -1$, we are able to absorb the term $\tau^{-1} ( \| \tilde{\phi}^{n+1} \|_2 -1)^2$ as an additional truncation error. This in turn makes an $\ell^\infty (0, T; H_h^2) \cap \ell^2 (0, T; H_h^3)$ optimal rate error estimate go through. 

In comparison, in the $H_h^2$ error estimate for the 3-D viscous Burgers' equation in~\cite{gottlieb12b}, 
the theoretical analysis  is much more straightforward, since there is no renormalization stage in the associated numerical design. Therefore, the convergence analysis for any PDE system with a nonlinear constraint (such as the $L^2$ norm-preservation) turns out to be more challenging than the one without it, which comes from the difficulty in the error estimate for the renormalization stage. 

\end{remark}

\section{Numerical results} \label{sec: numerical results} 
We present comprehensive numerical tests using the second-order centered finite difference spatial discretization to obtain the local minimum of   \eqref{eng_orig} subject to \eqref{cons} with periodic boundary conditions. The stopping criterion, $\frac{\Vert \phi^{n+1} - \phi^n \Vert_\infty}{\tau} \leq 10^{-8}$,  
ensures convergence, with all computations performed in MATLAB on standard hardware.

\subsection{One-dimensional numerical results}

\begin{example}\label{Ex_omg01D}
In this example, we consider the local minimum of   \eqref{eng_orig} subject to \eqref{cons}  on the interval $[-16, 16]$. The potential function and parameters are specified as
\begin{align}
V(x) = \frac{x^2}{2} + 25 \sin^2\left(\frac{\pi x}{4}\right), \quad \beta = 250, \quad \phi_0(x) = \exp\left(-\frac{|x|^2}{2}\right)/\pi^{1/4}.
\end{align}
\end{example}
\subsubsection{Energy stability test}
To examine the energy stability stated in Proposition~\ref{prop: energy stability}, different values of $A$ are tested in Figure~\ref{fig:DiffA}. It is observed that the choice of $A$ as
\begin{equation}\label{A_numval}
 A = A_{\text{ref}}:=\frac{3\beta}{2} \| \phi^n \|_\infty^2 
   + \frac12 \| V \|_\infty+1, 
\end{equation}
is sufficient to ensure energy stability, whereas smaller values may result in instability. In the following results, whenever not explicitly specified, the stabilizer is set as \eqref{A_numval}.

\begin{figure}[H]
\centering
    \includegraphics[width=0.6\textwidth]{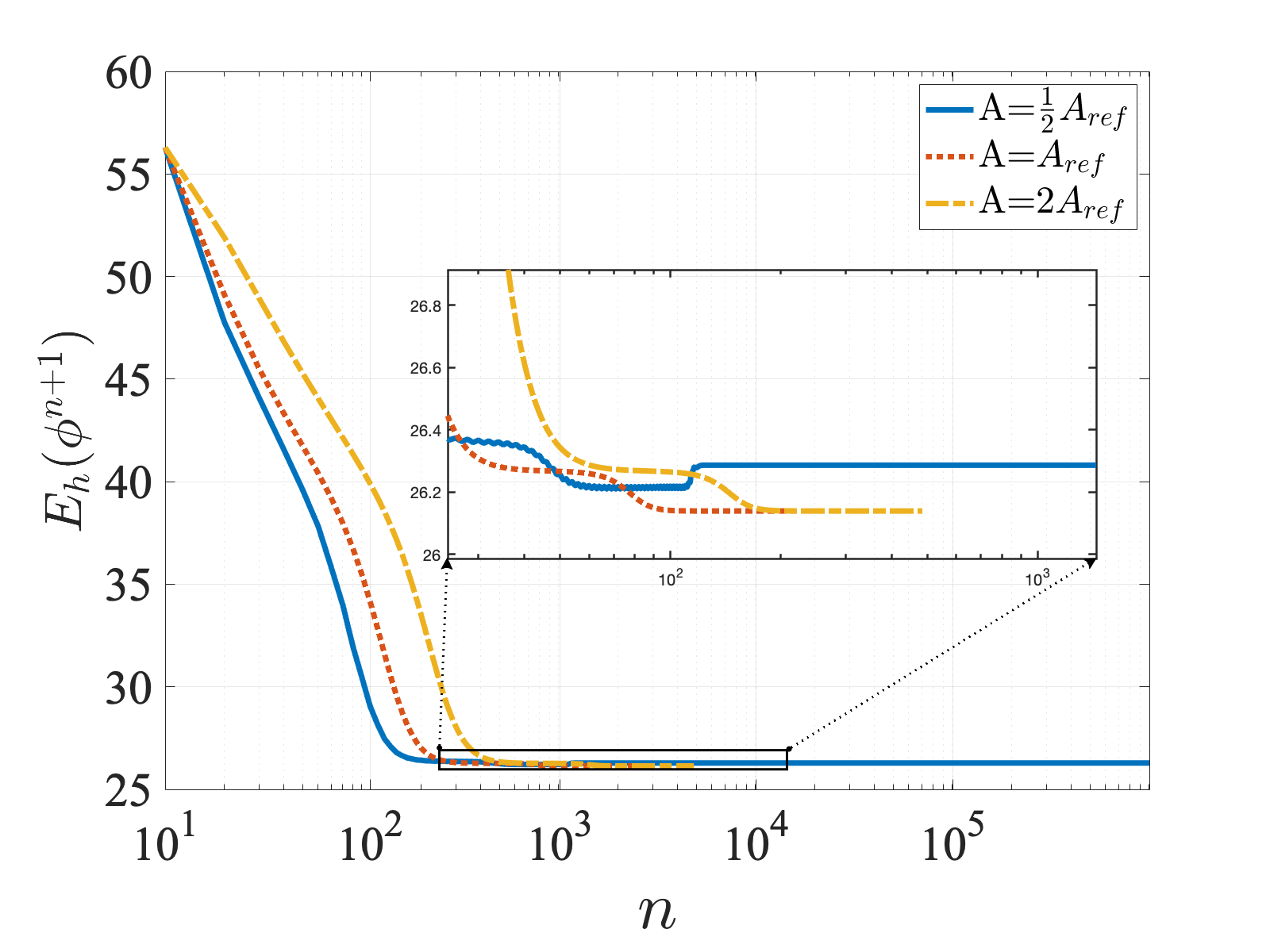}
\caption{Numerical energy evolution in terms of time step, with \(h=\frac{1}{128}\) and \(\tau =\frac{1}{4}\) for Example~\ref{Ex_omg01D}.}
\label{fig:DiffA}
\end{figure}

\subsubsection{Convergence test}
Figure~\ref{fig:1D_rate} illustrates the $\ell^2$ and $\ell^\infty$ numerical errors associated with scheme \eqref{num_scheme}. Here, $\phi_g$ denotes the numerical local minimum obtained under the stopping criterion, while $\phi_{g,\frac{\tau}{2}}^{\text{ref}}$ and $\phi_{g,\frac{h}{2}}^{\text{ref}}$ correspond to reference solutions computed with refined temporal and spatial step sizes $\tau/2$ and $h/2$, respectively. The results confirm that the temporal convergence rate is $\mathcal{O}(\tau)$ and the spatial convergence rate is $\mathcal{O}(h^2)$, in agreement with Theorem~\ref{thm:convergence}.

\begin{figure}[H]
\centering
\begin{subfigure}{0.45\textwidth}
    \includegraphics[width=\textwidth]{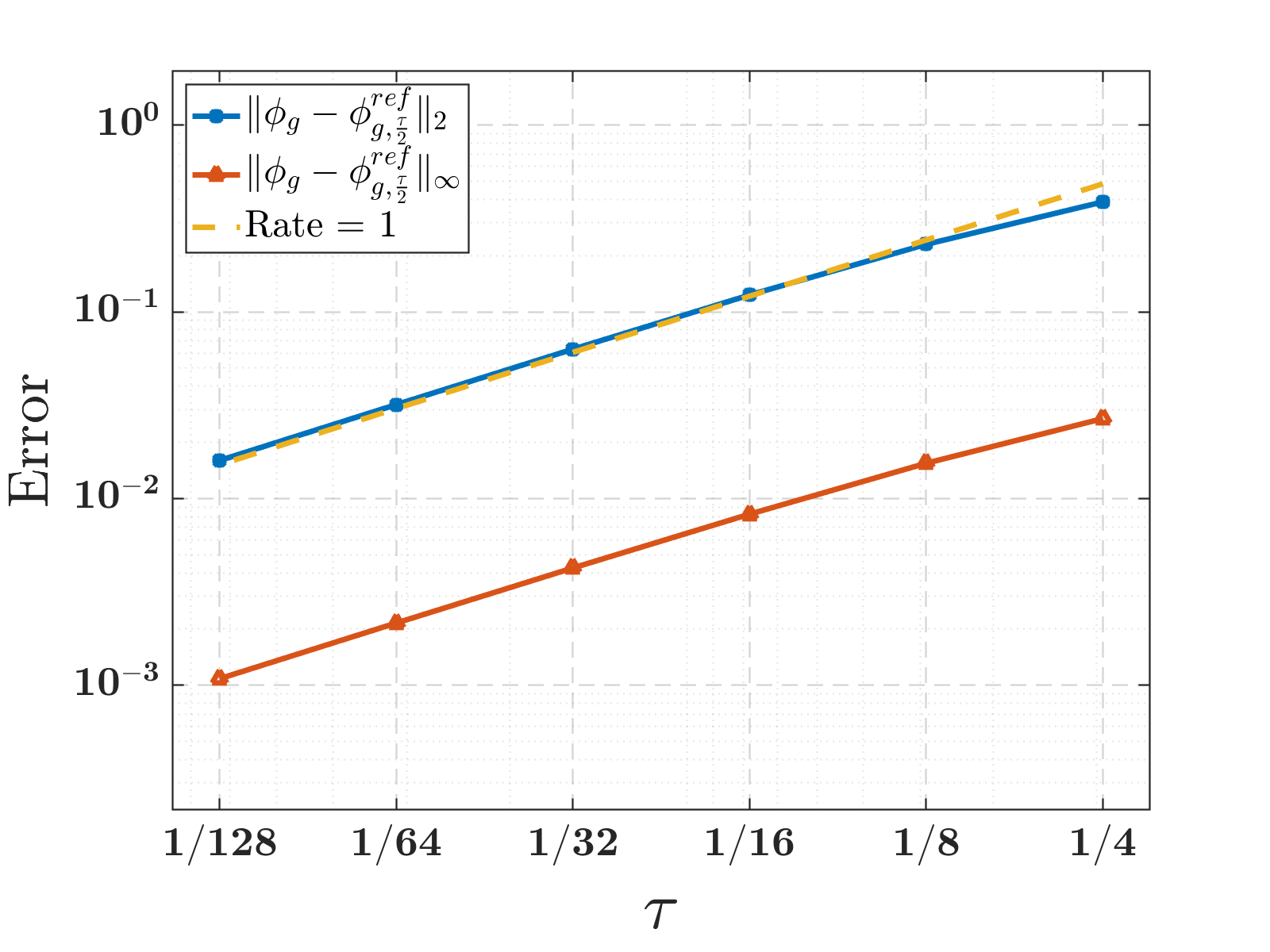}
    \caption{Temporal refinement with $h=\frac{1}{128}$}
\end{subfigure}
\hspace{-0.1cm}
\begin{subfigure}{0.45\textwidth}
    \includegraphics[width=\textwidth]{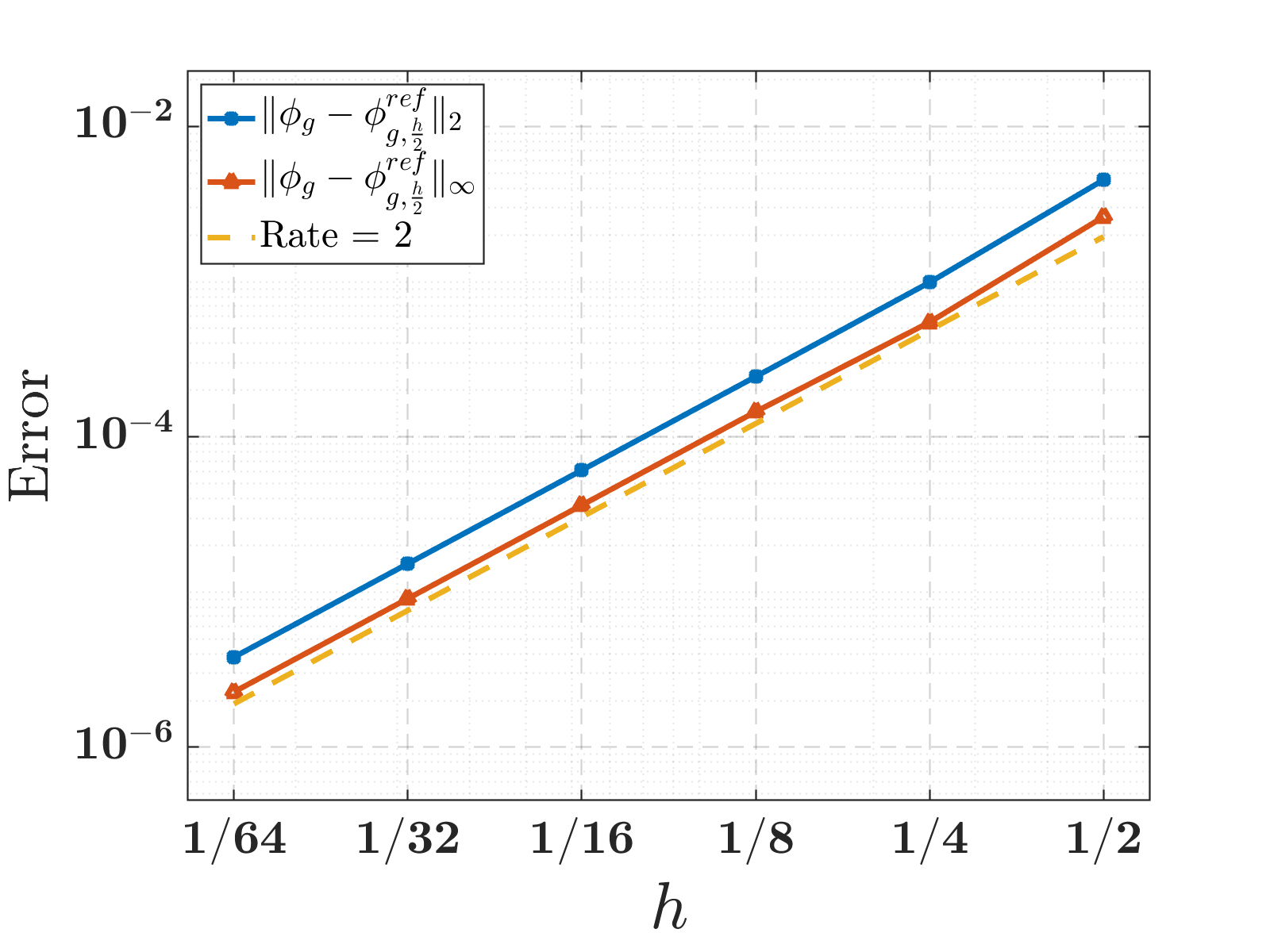}
    \caption{Spatial refinement with $\tau=\frac{1}{1000}$}
\end{subfigure}
\caption{Convergence rates for Example~\ref{Ex_omg01D}.}
\label{fig:1D_rate}
\end{figure}

\begin{figure}[H]
\centering
    \includegraphics[width=0.6\textwidth]{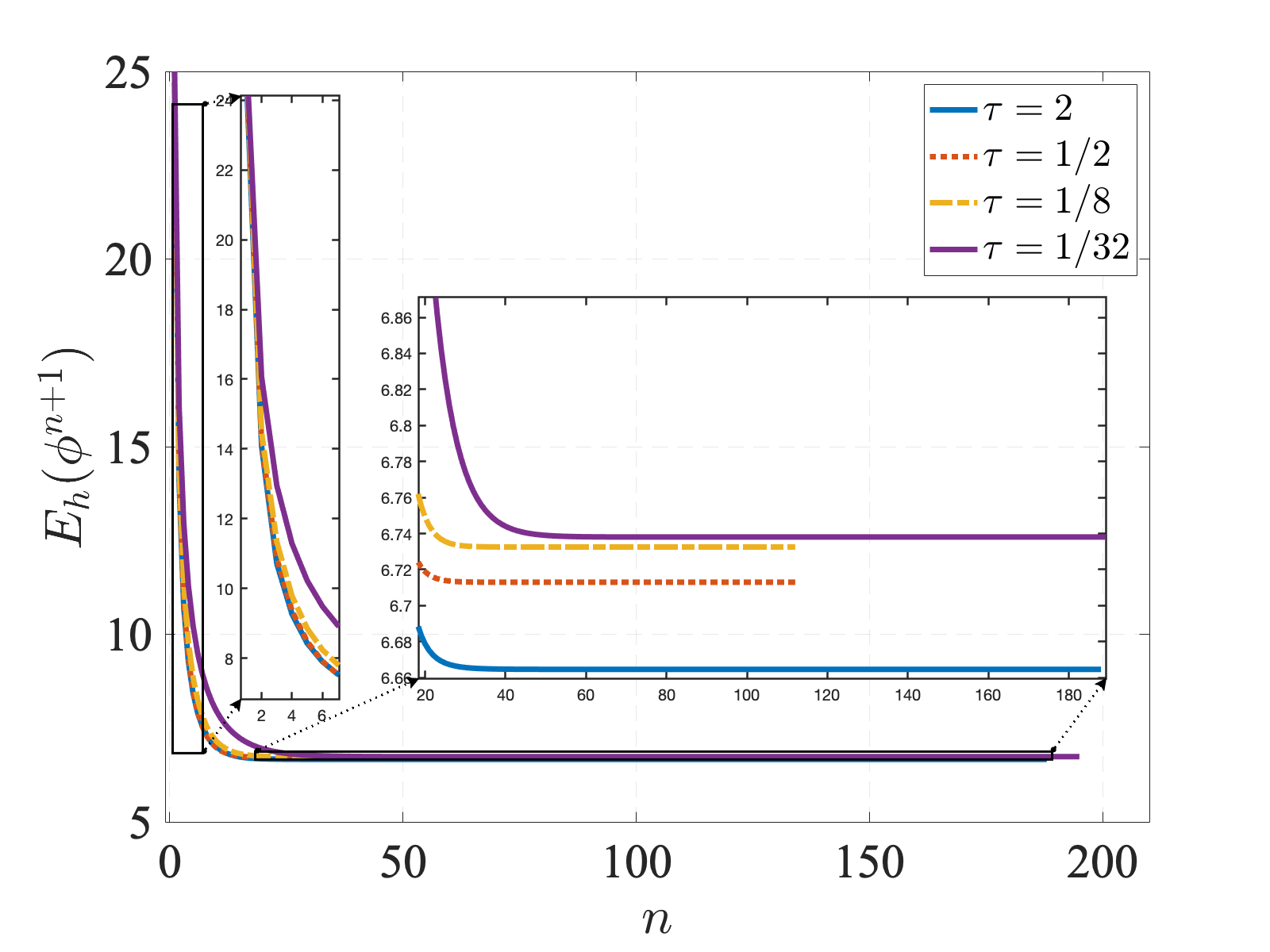}
\caption{Numerical energy evolution in terms of time step, with \(h=\frac{1}{16}\) for Example~\ref{Ex:2Domg0}.}
\label{fig:Diffdt2D}
\end{figure}

\subsection{Two-dimensional numerical results}
\begin{example}\label{Ex:2Domg0}\cite[Example 4.2]{YinHuaZha}
In this example, we consider the two-dimensional profile of \eqref{eng_orig} over the bounded domain $[-8,8]^2$, with $V({\bf x}) = \frac{1}{2}|{\bf x}|^2$,  $\beta = 300$. The initial condition is chosen as $\phi({\bf x}) = \frac{\e^{-V({\bf x})}}{\|\e^{-V({\bf x})}\|_2}$.
\end{example}

To investigate the energy stability of the numerical scheme \eqref{num_scheme} with the stabilizer parameter given by~\eqref{A_numval}, we perform simulations using several choices of the temporal step size $\tau$. The results, presented in Figure~\ref{fig:Diffdt2D}, indicate that the discrete energy is always non-increasing for all tested values of $\tau$, including relatively large time step sizes. This demonstrates that the stabilizer parameter given by \eqref{A_numval} is effective in maintaining energy stability in two-dimensional computations.

The convergence behavior of the scheme is further examined in Figure~\ref{fig:2D_rate}, where the numerical errors are measured against reference solutions obtained with refined temporal and spatial resolutions. The observed convergence rates confirm that the temporal error scales as $\mathcal{O}(\tau)$, while the spatial error scales as $\mathcal{O}(h^2)$, in agreement with the theoretical results established in Theorem~\ref{thm:convergence}. These results collectively demonstrate that scheme \eqref{num_scheme} is both stable and convergent.

\begin{figure}[H]
\centering
\begin{subfigure}{0.45\textwidth}
    \includegraphics[width=\textwidth]{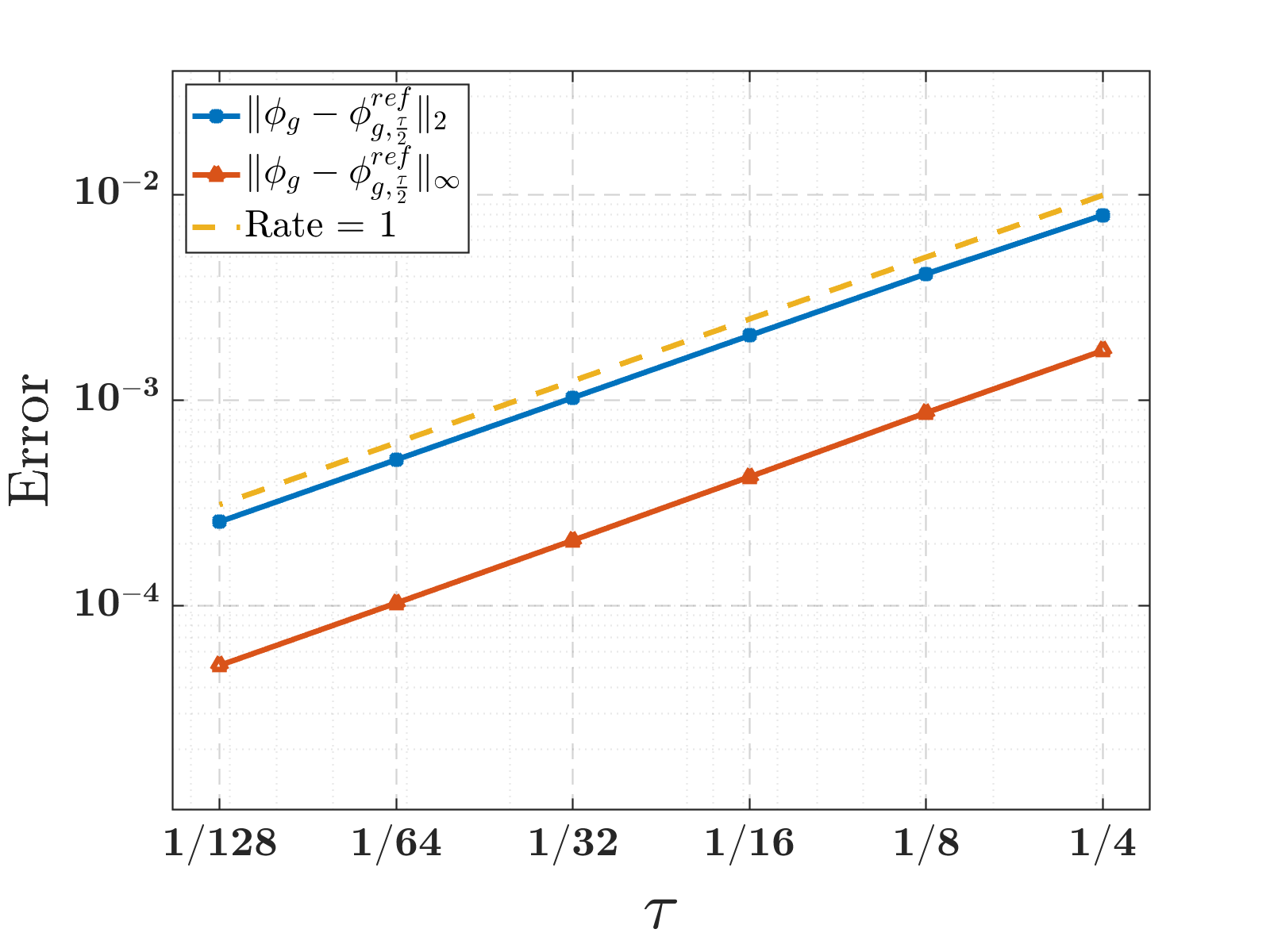}
    \caption{Temporal refinement with $h=\frac{1}{16}$}
\end{subfigure}
\hspace{-0.1cm}
\begin{subfigure}{0.45\textwidth}
    \includegraphics[width=\textwidth]{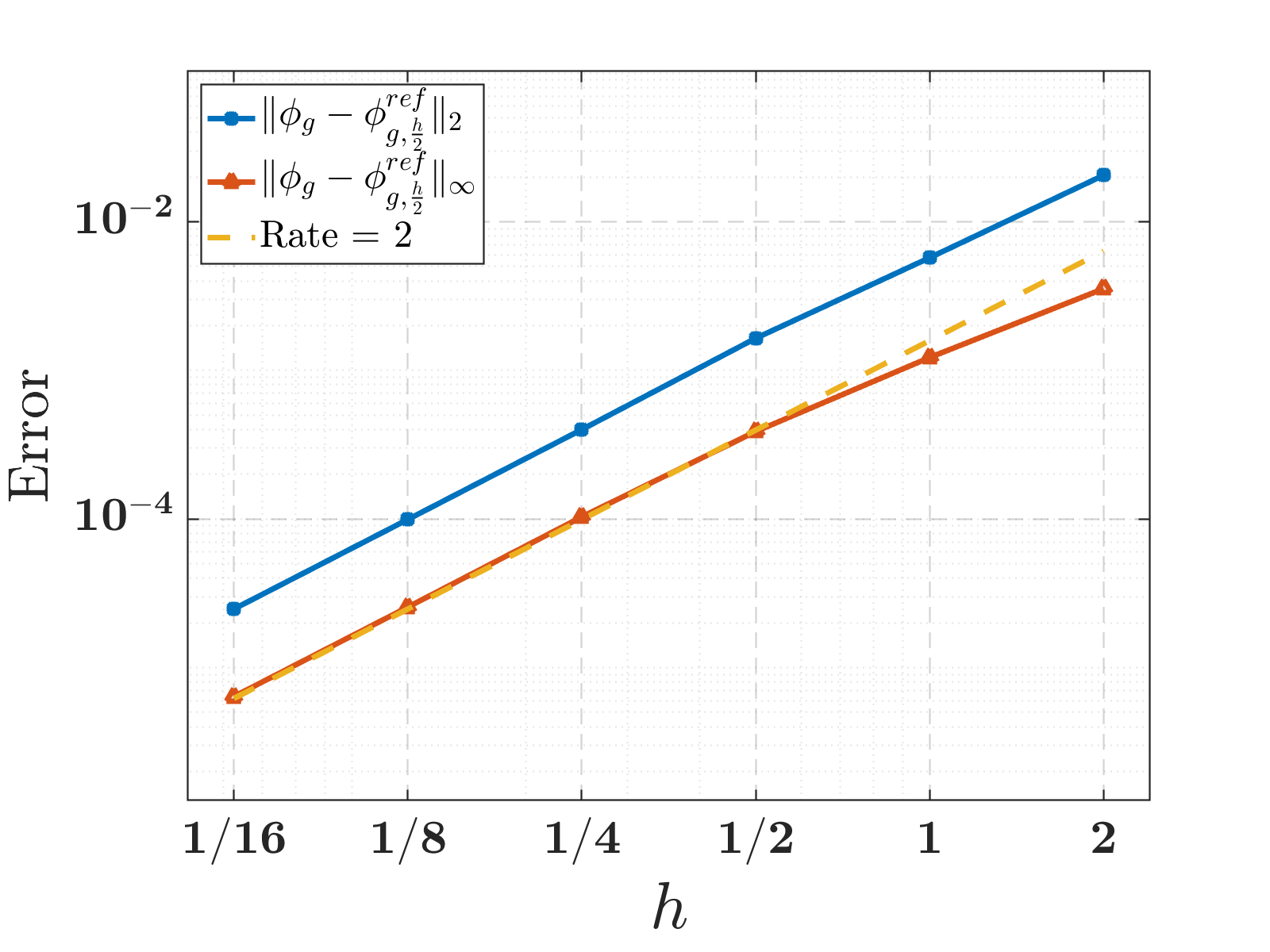}
    \caption{Spatial refinement with $\tau=\frac{1}{1000}$}
\end{subfigure}
\caption{Convergence rates for Example~\ref{Ex:2Domg0}.}
\label{fig:2D_rate}
\end{figure}

\section{Conclusions} \label{sec: conclusion}  

In this work, a normalized gradient flow approach, which contains a Lagrange multiplier to enforce the $L^2$ conservation, is used in the computation of the ground state Bose--Einstein condensates (BEC). The free energy dissipation comes from the variational structure of the PDE. An efficient numerical scheme is proposed and analyzed for such a normalized gradient equation. In the numerical design, an explicit approximation is applied to the nonlinear term and the external confinement term, an exponential time differencing (ETD) formula is used in the computation of the diffusion term, combined with an artificial regularization term. These numerical treatments create an intermediate profile, followed by an $L^2$ normalization at the next stage. The proposed numerical scheme has been proved to preserve the following theoretical properties: (1) an explicit computation at each time step, (2) unconditional free energy dissipation, (3) $L^2$ norm conservation at each time step, (4) a theoretical justification of convergence analysis and optimal rate error estimate. This turns out to be the first such work that preserves the combined four theoretical properties in the computation of the BEC problem. In addition, a few numerical examples are presented to validate these theoretical results, demonstrating excellent agreement with established reference solutions.

\appendix 

\section{Proof of Proposition~\ref{prop: convergence-prelim}} \label{appendix: prop 2} 

The derivation of the two inequalities in~\eqref{prop 1-1-1} is based on a detailed Fourier expansions of ${\cal G}_h^\frac12 \nabla_h f$ and ${\cal G}_h^\frac12 \Delta_h f$; see a similar proof in an existing work~\cite{LiX2025} (in Proposition 2.1). The technical details are skipped for the sake of brevity.  

Inequalities in~\eqref{prop 1-1-2} stand for a discrete Sobolev embedding from $H^2$ into $W^{1,4}$ and $L^\infty$, as well as a discrete elliptic regularity estimate. These inequalities have been proved in~\cite{guo16} (in Lemma 3.1), and we just cite the results here. 

Inequality~\eqref{prop 1-1-2-2} plays an important role in the nonlinear estimate. We begin with a nonlinear expansion in the finite difference approximation: 
\begin{equation} 
  D_x^2 (f g) = a_x (A_x f) D_x^2 g + 2 a_x (D_x f) a_x (D_x g) + a_x (A_x g) D_x^2 f . 
  \label{prop 1-NL-1} 
\end{equation} 
In turn, an application of discrete H\"older inequality implies that 
\begin{equation} 
  \| D_x^2 (f g) \|_2 \le \| f \|_\infty \cdot \| D_x^2 g \|_2 + 2 \| D_x f \|_4 \cdot \| D_x g \|_4 
   +\|  g \|_\infty  \| D_x^2 f \|_2 ,  
  \label{prop 1-NL-2} 
\end{equation} 
in which an obvious fact, $\| A_x \varphi \|_2 \le \| \varphi \|_2$, $\| a_x \psi \|_2 \le \| \psi \|_2$, has been applied. Similar estimates could be derived for $D_y^2 (fg)$ and $D_z^2 (fg)$: 
\begin{equation} 
\begin{aligned} 
  & 
  \| D_y^2 (f g) \|_2 \le \| f \|_\infty \cdot \| D_y^2 g \|_2 + 2 \| D_y f \|_4 \cdot \| D_y g \|_4 
   +\|  g \|_\infty  \| D_y^2 f \|_2 ,  
\\
  & 
  \| D_z^2 (f g) \|_2 \le \| f \|_\infty \cdot \| D_z^2 g \|_2 + 2 \| D_z f \|_4 \cdot \| D_z g \|_4 
   +\|  g \|_\infty  \| D_z^2 f \|_2 . 
\end{aligned} 
  \label{prop 1-NL-3} 
\end{equation}  
Then we arrive at 
\begin{equation} 
\begin{aligned} 
  & 
  \| \Delta_h (f g) \|_2 \le \| f \|_\infty ( \| D_x^2 g \|_2 + \| D_y^2 g \|_2 + \| D_z^2 g \|_2 ) 
  + 2 \| \nabla_h f \|_4 \cdot \| \nabla_h g \|_4 
\\
  & 
   + \|  g \|_\infty  ( \| D_x^2 f \|_2 + \| D_y^2 f \|_2 + \| D_z^2 f \|_2 )  ,  
\\
  \le & 
  \sqrt{3} \breve{C}_0 ( \| f \|_\infty ( \| g \|_2 + \| \Delta_h g \|_2 ) 
  + \| g \|_\infty ( \| f \|_2 + \| \Delta_h f \|_2 ) ) 
  + 2 \| \nabla_h f \|_4 \cdot \| \nabla_h g \|_4 
\\
  \le & 
  \sqrt{3} \breve{C}_0^2 ( \| f \|_2 + \| \Delta_h f \|_2 )  ( \| g \|_2 + \| \Delta_h g \|_2 ) ) 
  + 2 \breve{C}_0^2 \| \Delta_h f \|_2 \cdot \| \Delta_h g \|_2 
\\
  \le & 
  (\sqrt{3} +2)  \breve{C}_0^2 ( \| f \|_2 + \| \Delta_h f \|_2 )  ( \| g \|_2 + \| \Delta_h g \|_2 ) ) , 
\end{aligned} 
  \label{prop 1-NL-4} 
\end{equation}  
in which the discrete Sobolev inequalities \eqref{prop 1-1-2} have been extensively applied. This has proved inequality \eqref{prop 1-1-2-2}, by taking $\breve{C}_1 = (\sqrt{3} +2) \breve{C}_0^2$.

In terms of the $\| \cdot \|_2$ estimate of ${\cal NLE}^n$, we see that the $\ell^\infty$ bound for the approximate and numerical solutions (given by~\eqref{exact-inf-1} and \eqref{a priori-3}), implies that 
\begin{equation} 
\begin{aligned}
  & 
  \| | \phi^n |^2 e^n \|_2 \le \| \phi^n \|_\infty^2 \cdot \| e^n \|_2 \le \tilde{C}_1^2 \| e^n \|_2 , 
\\
  & 
   \| ( (\Phi^n)^c e^n + \phi^n ( e^n )^c ) \Phi^n \|_2 
   \le (\| \Phi^n \|_\infty + \| \phi^n \|_\infty ) \| e^n \|_2  \cdot \| \Phi^n \|_\infty  
\\
  &  \qquad \qquad \qquad \qquad \qquad \quad \, \, 
   \le C^* ( C^* + \tilde{C}_1 ) \| e^n \|_2 , 
\\
  & \mbox{so that} \quad 
  \| {\cal NLE}^n \|_2 \le \| | \phi^n |^2 e^n \|_2  + \| ( (\Phi^n)^c e^n + \phi^n ( e^n )^c ) \Phi^n \|_2   
  \le \tilde{C}_3 \| e^n \|_2 , 
\end{aligned} 
  \label{prop 1-2} 
\end{equation} 
in which $\tilde{C}_3 = \tilde{C}_1^2 + C^* ( C^* + \tilde{C}_1 ) $. 

Regarding the $H_h^2$ estimate for ${\cal NLE}^n$, we begin with the following observation, which turns out to be an application of inequality~\eqref{prop 1-1-2-2}, as well as the a-priori estimate~\eqref{a priori-3}:
\begin{equation} 
\begin{aligned} 
  & 
  \| | \phi^n |^2 \|_2 \le \| \phi^n \|_\infty \cdot \| \phi^n \|_2  \le \tilde{C}_1^2 | \Omega|^\frac12 ,   
\\
  & 
   \| \Delta_h (| \phi^n |^2 ) \|_2 \le \breve{C}_1 ( \| \phi^n \|_2 + \| \Delta_h \phi^n \|_2 )^2 
   \le \breve{C}_1 \tilde{C}_1^2 ( 1 + | \Omega |^\frac12 )^2 .  
\end{aligned} 
  \label{prop 1-3} 
\end{equation} 
Subsequently, a further application of~\eqref{prop 1-1-2-2} reveals that 
\begin{equation} 
\begin{aligned} 
  \| \Delta_h (| \phi^n |^2 e^n )  \| \le & \breve{C}_1 (  \| | \phi^n |^2 \|_2 + \| \Delta_h (| \phi^n |^2 ) \|_2 ) 
  ( \| e^n \|_2 + \| \Delta_h e^n \|_2)  
\\ 
  \le & 
  \breve{C}_1  (\breve{C}_1 + 1) \tilde{C}_1^2 ( 1 + | \Omega |^\frac12 )^2
  ( \| e^n \|_2 + \| \Delta_h e^n \|_2) . 
\end{aligned} 
  \label{prop 1-4} 
\end{equation} 
The other part of $\Delta_h {\cal NLE}^n$ could be similarly analyzed, and the technical details are skipped for the sake of brevity: 
\begin{equation} 
\begin{aligned} 
  & 
  \| \Delta_h ( (\Phi^n)^c e^n + \phi^n ( e^n )^c ) \Phi^n \|_2  
\\
  \le & 
  \breve{C}_1 (\breve{C}_1 + 1)  C^* ( \tilde{C}_1 + C^* ) ( 1 + | \Omega |^\frac12 )^2 
  ( \| e^n \|_2 + \| \Delta_h e^n \|_2) . 
\end{aligned} 
  \label{prop 1-5} 
\end{equation} 
In turn, a combination of \eqref{prop 1-4} and \eqref{prop 1-5} results in the desired gradient estimate: 
\begin{equation} 
\begin{aligned} 
  & 
   \| \Delta_h {\cal NLE}^n \|_2 \le  \| \Delta_h ( | \phi^n |^2 e^n ) \|_2 
   +  \| \Delta_h ( (\Phi^n)^c e^n + \phi^n ( e^n )^c ) \Phi^n \|_2  
\\
  \le & 
   \breve{C}_1 (\breve{C}_1 + 1)  ( (C^*)^2 + C^* \tilde{C}_1 + \tilde{C}_1^2 ) 
   ( 1 + | \Omega |^\frac12 )^2 ( \| e^n \|_2 + \| \Delta_h e^n \|_2) . 
\end{aligned} 
  \label{prop 1-6} 
\end{equation} 
This has proved the second inequality in~\eqref{prop 1-1-3}, by taking with $\tilde{C}_3 =  \breve{C}_1 (\breve{C}_1 + 1)  ( (C^*)^2 + C^* \tilde{C}_1 + \tilde{C}_1^2 )  ( 1 + | \Omega |^\frac12 )^2$. 

To obtain an upper bound for the Lagrange multiplier $\lambda^n$, we see that a combination of the representation formula~\eqref{lambda_disc} and the energy stability estimate~\eqref{ener stab-0} (for the numerical solution) reveals that 
\begin{equation}  
  0 \le \lambda^n 
  = \frac12 \| {\cal G}_h^\frac12 \nabla_h \phi^n \|_2^2  
  + \langle V({\bf x}) \phi^n , \phi^n \rangle 
  + \beta \langle |\phi^n|^2 \phi^n, \phi^n \rangle  \le 2 E_h (\phi^n ) \le 2 C_0 . 
  \label{prop 1-7} 
\end{equation} 
Therefore, the first inequality in~\eqref{prop 1-1-4} has been proved, by taking $\tilde{C}_4 = 2 C_0$. In terms of the error estimate for the Lagrange multiplier, the following bounds are observed:  
\begin{equation} 
\begin{aligned} 
  & 
  \| {\cal G}_h^\frac12 \nabla_h \Phi^n \|_2 \le \| \nabla_h \Phi^n \|_2 \le C^* | \Omega |^\frac12 , \quad  
  \| {\cal G}_h^\frac12 \nabla_h \phi^n \|_2  \le ( 2 C_0)^\frac12 ,  \, \, \, 
  \mbox{(by~\eqref{ener stab-H1-1})} , 
\\
  & 
  | \langle {\cal G}_h^\frac12 \nabla_h ( \Phi^n + \phi^n ) ,   {\cal G}_h^\frac12 \nabla_h e^n \rangle | 
  \le ( \| {\cal G}_h^\frac12 \nabla_h \Phi^n \|_2  + \| {\cal G}_h^\frac12 \nabla_h \phi^n \|_2 ) 
  \| {\cal G}_h^\frac12 \nabla_h e^n \|_2 
\\
  \le & 
  ( C^* | \Omega |^\frac12 + ( 2 C_0)^\frac12 )  \| {\cal G}_h^\frac12 \nabla_h e^n \|_2 , 
\\
  & 
  | \langle V(\mathbf{x}) ( \Phi^n + \phi^n ) , e^n \rangle |  
  \le \| V \|_\infty ( \| \Phi^n \|_\infty + \| \phi^n \|_2 ) \| e^n \|_2 
\\
  \le & 
    | \Omega |^\frac12 \| V \|_\infty ( C^* + \tilde{C}_1 ) \| e^n \|_2 , 
\\
  & 
  | \langle |\phi^n|^2 \phi^n, e^n \rangle | 
  \le \| \phi^n \|_\infty^2 \cdot \| \phi^n \|_2 \cdot \| e^n \|_2 
  \le \tilde{C}_1^3 | \Omega |^\frac12 \| e^n \|_2 , 
\\
  & 
  | \langle {\cal NLE}^n , \Phi^n \rangle | 
  \le  \| {\cal NLE}^n \|_2 \cdot \| \Phi^n \|_2  
  \le \tilde{C}_2 \| e^n \|_2 \cdot C^* | \Omega|^\frac12 
  \le \tilde{C}_2 C^* | \Omega|^\frac12 \| e^n \|_2 . 
\end{aligned} 
  \label{prop 1-8} 
\end{equation} 
As a consequence, we arrive at the following estimate 
\begin{equation} 
\begin{aligned} 
  | e_\lambda^n | \le & \frac12 | \langle {\cal G}_h^\frac12 \nabla_h ( \Phi^n + \phi^n ) , 
   {\cal G}_h^\frac12 \nabla_h e^n \rangle  | 
   + | \langle V(\mathbf{x}) ( \Phi^n + \phi^n ) , e^n \rangle | 
\\
   & 
   + \beta ( | \langle |\phi^n|^2 \phi^n, e^n \rangle | 
   + | \langle {\cal NLE}^n , \Phi^n \rangle | ) 
   \le \tilde{C}_5 ( \| e^n \|_2 + \| {\cal G}_h^\frac12 \nabla_h e^n \|_2 ) , 
\end{aligned} 
  \label{prop 1-9} 
\end{equation} 
with $\tilde{C}_5 = \max (  \frac12 ( C^* | \Omega |^\frac12 + ( 2 C_0)^\frac12 ) ,  
| \Omega |^\frac12 ( \| V \|_\infty ( C^* + \tilde{C}_1 ) + \beta ( \tilde{C}_1^3  
+ \tilde{C}_2 C^* ) )$. The proof of Proposition~\ref{prop: convergence-prelim} is completed. 
  

\section{Proof of Proposition~\ref{prop: renormalization}} \label{appendix: prop 3} 


Based on a rewritten form~\eqref{lem 1-5} of the numerical solution at the intermediate stage, we see that 
\begin{equation} 
\begin{aligned} 
  & 
  \| \mathcal{G}_h \Delta_h \phi^n  \| \le \| \Delta_h \phi^n \|_2 \le \tilde{C}_1 , \quad 
  \mbox{(by~\eqref{a priori-3})} , 
\\
  & 
  \| V({\bf x}) \phi^n \|_2 \le \| V \|_\infty \cdot \| \phi^n \|_2 \le C^* \cdot 1 = C^* ,  \quad 
  \| \lambda^n \phi^n \|_2  = \lambda^n \| \phi^n \|_2 \le \lambda^n \le \tilde{C}_4 , 
\\
  & 
  \| |\phi^n|^2\phi^n \|_2 \le \| \phi^n \|_\infty^2 \cdot \| \phi^n \|_2 
  \le \tilde{C}_1^2 \cdot 1 = \tilde{C}_1^2 ,  \quad \mbox{so that} 
\\
  & 
  \| \tilde{\phi}^{n+1} - \phi^n \|_2 \le \frac{1}{\frac{1}{\tau} +A}  
   \Big( \frac{1}{2} \| \mathcal{G}_h \Delta_h \phi^n \|_2 
 + \| V({\bf x})\phi^n \|_2 + \beta \| |\phi^n|^2\phi^n \|_2 + \lambda^n \| \phi^n \|_2 \Big) 
\\
  & \qquad \qquad \qquad  
  \le \hat{C}_1 \tau , \quad \mbox{with} \, \, \, 
  \hat{C}_1 = \frac{\tilde{C}_1}{2} + C^* + \tilde{C}_4 + \beta \tilde{C}_1^2 .   
\end{aligned} 
  \label{L2 est-1} 
\end{equation} 
Notice that $\hat{C}_1$ only depends on the exact solution. Meanwhile, by the $\ell^2$ orthogonality~\eqref{lem 1-1-2} between $\phi^n$ and $\tilde{\phi}^{n+1} - \phi^n$, an $\ell^2$ inner product expansion reveals that 
\begin{equation} 
\begin{aligned} 
  1 \le \| \tilde{\phi}^{n+1} \|_2^2 = & \| \phi^n \|_2^2 + \| \tilde{\phi}^{n+1} - \phi^n \|_2^2 
  + 2 \langle \tilde{\phi}^{n+1} - \phi^n , \phi^n \rangle 
\\
  = & 
   1 + \| \tilde{\phi}^{n+1} - \phi^n \|_2^2 
 \le  1 + \hat{C}_1^2 \tau^2 .  
\end{aligned} 
  \label{L2 est-2} 
\end{equation} 
As a consequence, taking a square root on both sides leads to 
\begin{equation} 
  1 \le \| \tilde{\phi}^{n+1} \|_2 \le  ( 1 + \hat{C}_1^2 \tau^2 )^\frac12 
  \le  1 + \frac{\hat{C}_1^2}{2} \tau^2 , 
  \label{L2 est-3} 
\end{equation} 
in which inequality $\sqrt{1 +a} \le 1 + \frac{a}{2}$ (for $a \ge 0)$) has been applied. This has proved inequality~\eqref{prop 2-1-0}, by taking $\tilde{C}_6 = \frac12 \hat{C}_1^2$. 

Moreover, the renormalization formula $\phi^{n+1} = \frac{\tilde{\phi}^{n+1}}{\| \tilde{\phi}^{n+1} \|_2}$ implies the following facts: 
\begin{equation} 
\begin{aligned} 
  & 
  \tilde{\phi}^{n+1} - \phi^{n+1} = ( \| \tilde{\phi}^{n+1} \|_2 -1 ) \phi^{n+1} ,  \quad 
  \| \tilde{\phi}^{n+1} - \phi^{n+1} \|_2 =  \| \tilde{\phi}^{n+1} \|_2 -1 , 
\\
  & 
  \tilde{e}^{n+1} = e^{n+1} - (  \tilde{\phi}^{n+1} - \phi^{n+1} ) 
  =  e^{n+1} - ( \| \tilde{\phi}^{n+1} \|_2 -1 ) \phi^{n+1} . 
\end{aligned} 
   \label{prop 2-2} 
\end{equation} 
Meanwhile, the discrete functions $\phi^{n+1}$, $\Phi^{n+1}$ and $e^{n+1} = \Phi^{n+1} - \phi^{n+1}$ form a functional triangle, with $\| \phi^{n+1} \|_2 = \| \Phi^{n+1} \|_2 =1$. Furthermore, the identity that $\Phi^{n+1} = \phi^{n+1} + e^{n+1}$ indicates the following observation 
\begin{equation} 
\begin{aligned} 
  & 
  \| \Phi^{n+1} \|_2^2 = \| \phi^{n+1} \|_2^2 + \| e^{n+1} \|_2^2 + 2 \langle \phi^{n+1} , e^{n+1} \rangle ,  
\\
  & \mbox{so that} \quad 
  \langle \phi^{n+1} , e^{n+1} \rangle = - \frac12 \| e^{n+1} \|_2^2 \le 0 ,  \quad 
  \mbox{since} \, \, \, \| \phi^{n+1} \|_2 = \| \Phi^{n+1} \|_2 =1 . 
\end{aligned} 
  \label{prop 2-3} 
\end{equation} 
Subsequently, its combination with~\eqref{prop 2-2} leads to the following estimate: 
\begin{equation} 
\begin{aligned} 
  \| \tilde{e}^{n+1} \|_2^2 = & \| e^{n+1} \|_2^2 + ( \| \tilde{\phi}^{n+1} \|_2 -1 )^2 \| \phi^{n+1} \|_2^2 
  - 2 ( \| \tilde{\phi}^{n+1} \|_2 -1 )  \langle  e^{n+1} , \phi^{n+1} \rangle 
\\
  \ge & 
   \| e^{n+1} \|_2^2 + ( \| \tilde{\phi}^{n+1} \|_2 -1 )^2 ,  
\end{aligned} 
   \label{prop 2-4} 
\end{equation} 
in which the fact that $\| \phi^{n+1} \|_2 =1$ and the a-priori estimate $\| \tilde{\phi}^{n+1} \|_2 \ge 1$ have been applied in the derivation. This proves the left inequality in~\eqref{prop 2-1-1}. The right inequality of \eqref{prop 2-1-1} comes from a direct application of Cauchy inequality: 
\begin{equation} 
\begin{aligned} 
  \| \tilde{e}^{n+1} \|_2^2 = & \| e^{n+1} \|_2^2 + ( \| \tilde{\phi}^{n+1} \|_2 -1 )^2 \| \phi^{n+1} \|_2^2 
  - 2 ( \| \tilde{\phi}^{n+1} \|_2 -1 )  \langle  e^{n+1} , \phi^{n+1} \rangle 
\\
  \le & 
   2 ( \| e^{n+1} \|_2^2 + ( \| \tilde{\phi}^{n+1} \|_2 -1 )^2 ) .   
\end{aligned} 
   \label{prop 2-5} 
\end{equation} 

The estimate of $\Delta_h e^{n+1}$ in terms of $\Delta_h \tilde{e}^{n+1}$ turns out to be more challenging. Based on a simple fact, $\Delta_h \phi^{n+1} = \frac{\Delta_h \tilde{\phi}^{n+1}}{\| \tilde{\phi}^{n+1} \|_2}$, the following expansion is observed:  
\begin{equation} 
\begin{aligned} 
  \Delta_h e^{n+1} = & \Delta_h \Phi^{n+1} - \Delta_h \phi^{n+1} 
  = \Delta_h \Phi^{n+1} - \Delta_h \tilde{\phi}^{n+1} 
  + \frac{( \| \tilde{\phi}^{n+1} \|_2 -1) \Delta_h \tilde{\phi}^{n+1}}{\| \tilde{\phi}^{n+1} \|_2} 
\\
  = & 
   \Delta_h \tilde{e}^{n+1} + ( \| \tilde{\phi}^{n+1} \|_2 -1) \Delta_h \phi^{n+1} . 
\end{aligned} 
  \label{prop 2-6-1} 
\end{equation} 
Meanwhile, the a-priori assumption~\eqref{a priori-4} indicates that 
\begin{equation} 
  \| \Delta_h \tilde{\phi}^{n+1} \|_2 \le \| \Delta_h \Phi^{n+1} \|_2  + \| \Delta_h \tilde{e}^{n+1} \|_2  
  \le C^* + 2 ( \tau^\frac38 + h^\frac54) \le C^* + \frac12 = \tilde{C}_1 , 
  \label{prop 2-6-2} 
\end{equation} 
provided that $\tau$ and $h$ are sufficiently small. Furthermore, we see that 
\begin{equation} 
  \| \Delta_h \phi^{n+1} \|_2 = \frac{\| \Delta_h \tilde{\phi}^{n+1} \|_2}{\| \tilde{\phi}^{n+1} \|_2} 
  \le \frac{\tilde{C}_1}{\| \tilde{\phi}^{n+1} \|_2} \le \tilde{C}_1 , \quad 
  \mbox{since} \, \, \, \, \| \tilde{\phi}^{n+1} \|_2 \ge 1 . 
  \label{prop 2-6-3} 
\end{equation} 
Subsequently, a discrete $\ell^2$ inner expansion is applied to~\eqref{prop 2-6-1}: 
\begin{equation} 
\begin{aligned} 
  & 
  \langle \Delta_h \tilde{e}^{n+1} , \Delta_h \phi^{n+1} \rangle  
  \le \| \Delta_h \tilde{e}^{n+1} \|_2 \cdot \| \Delta_h \phi^{n+1} \|_2 
  \le \tilde{C}_1  \| \Delta_h \tilde{e}^{n+1} \|_2 , 
\\
  & 
  2 ( \| \tilde{\phi}^{n+1} \|_2 -1) \langle \Delta_h \tilde{e}^{n+1} , \Delta_h \phi^{n+1} \rangle 
  \le 2 \tilde{C}_1 ( \| \tilde{\phi}^{n+1} \|_2 -1) \| \Delta_h \tilde{e}^{n+1} \|_2 
\\
  & \qquad \qquad \qquad \qquad \qquad 
  \le \tau \| \Delta_h \tilde{e}^{n+1} \|_2^2  
  + 4 \tilde{C}_1^2 \tau^{-1} ( \| \tilde{\phi}^{n+1} \|_2 -1)^2 , 
\\
  & 
  \| \Delta_h e^{n+1}\|_2^2  =   
  \|  \Delta_h \tilde{e}^{n+1} \|_2^2 + ( \| \tilde{\phi}^{n+1} \|_2 -1)^2 \| \Delta_h \phi^{n+1}\|_2^2  
\\
  & \qquad \qquad \qquad 
   + 2 ( \| \tilde{\phi}^{n+1} \|_2 -1) \langle \Delta_h \tilde{e}^{n+1} , \Delta_h \phi^{n+1} \rangle 
\\
  & \qquad \qquad \quad 
  \le \|  \Delta_h \tilde{e}^{n+1} \|_2^2 + \tilde{C}_1^2 ( \| \tilde{\phi}^{n+1} \|_2 -1)^2  
  + \tau \| \Delta_h \tilde{e}^{n+1} \|_2^2  
\\
  & \qquad \qquad \qquad 
  + 4 \tilde{C}_1^2 \tau^{-1} ( \| \tilde{\phi}^{n+1} \|_2 -1)^2 
\\
  & \qquad \qquad \quad 
  \le ( 1+ \tau) \|  \Delta_h \tilde{e}^{n+1} \|_2^2  
  + 5 \tilde{C}_1^2 \tau^{-1} ( \| \tilde{\phi}^{n+1} \|_2 -1)^2  , 
\end{aligned} 
  \label{prop 2-6-4} 
\end{equation} 
provided that $\tau$ is sufficiently small. Therefore, inequality~\eqref{prop 2-1-2} has been proved, by taking $\tilde{C}_6 =  5 \tilde{C}_1^2$. In addition, inequality of \eqref{prop 2-1-3} comes from a direct application of Cauchy inequality, by taking $\tilde{C}_7 = 2 \tilde{C}_1^2$:  
\begin{equation} 
\begin{aligned} 
  & 
  \Delta_h \tilde{e}^{n+1}\|_2  =   
   \Delta_h e^{n+1}  - ( \| \tilde{\phi}^{n+1} \|_2 -1) \Delta_h \phi^{n+1} ,  \quad \mbox{so that} 
\\
  & 
  \| \Delta_h \tilde{e}^{n+1}\|_2^2  \le   
  2 ( \|  \Delta_h e^{n+1} \|_2^2 + ( \| \tilde{\phi}^{n+1} \|_2 -1)^2 \| \Delta_h \phi^{n+1}\|_2^2 ) 
\\
  & \qquad \qquad \quad 
  \le 2 ( \|  \Delta_h e^{n+1} \|_2^2 + \tilde{C}_1^2 ( \| \tilde{\phi}^{n+1} \|_2 -1)^2 ) .   
\end{aligned} 
  \label{prop 2-6-5} 
\end{equation} 
This finishes the proof of Proposition~\ref{prop: renormalization}.

\end{document}